\documentclass[platex]{amsart}
\usepackage{amssymb}
\usepackage{braket}
\usepackage{mathrsfs}
\usepackage{ifthen}
\usepackage{graphicx}

\usepackage{tikz}
\usetikzlibrary{patterns}
\usepackage{comment} 
\usepackage{mleftright}
\usepackage{hyperref}
\usepackage{cleveref}
 \usepackage[all]{xy}
 \usepackage{amscd}
 \usepackage{amsrefs}
 \usepackage{color}

\theoremstyle{plain}
\newtheorem{theo}{Theorem}[section]
\crefname{theo}{Theorem}{Theorems}
\Crefname{theo}{Theorem}{Theorems}
\newtheorem{prop}[theo]{Proposition}
\crefname{prop}{Proposition}{Propositions}
\Crefname{prop}{Proposition}{Propositions}
\newtheorem{lem}[theo]{Lemma}
\crefname{lem}{Lemma}{Lemmas}
\Crefname{lem}{Lemma}{Lemmas}
\newtheorem{cor}[theo]{Corollary}
\crefname{cor}{Corollary}{Corollaries}
\Crefname{cor}{Corollary}{Corollaries}

\crefname{claim}{Claim}{Claims}
\Crefname{claim}{Claim}{Claims}

\crefname{property}{Property}{Properties}
\Crefname{property}{Property}{Properties}
\newtheorem{problem}[theo]{Problem}
\crefname{problem}{Problem}{Problems}
\Crefname{problem}{Problem}{Problems}

\theoremstyle{definition}
\newtheorem{defi}[theo]{Definition}
\crefname{defi}{Definition}{Definitions}
\Crefname{defi}{Definition}{Definitions}

\crefname{notation}{Notation}{Notations}
\Crefname{notation}{Notation}{Notations}

\crefname{convention}{Convention}{Conventions}
\Crefname{convention}{Convention}{Conventions}

\crefname{cond}{Condition}{Conditions}
\Crefname{cond}{Condition}{Conditions}
\newtheorem{assum}[theo]{Assumption}
\crefname{assum}{Assumption}{Assumptions}
\Crefname{assum}{Assumption}{Assumptions}

\theoremstyle{remark}
\newtheorem{rem}[theo]{Remark}
\crefname{rem}{Remark}{Remarks}
\Crefname{rem}{Remark}{Remarks}
\newtheorem{ex}[theo]{Example}
\crefname{ex}{Example}{Examples}
\Crefname{ex}{Example}{Examples}

\crefname{section}{Section}{Sections}
\Crefname{section}{Section}{Sections}
\crefname{subsection}{Subsection}{Subsections}
\Crefname{subsection}{Subsection}{Subsections}
\crefname{figure}{Figure}{Figures}
\Crefname{figure}{Figure}{Figures}

\newtheorem*{acknowledgement}{Acknowledgement}

\newcommand{\Z}{\mathbb{Z}}

\newcommand{\R}{\mathbb{R}}
\newcommand{\C}{\mathbb{C}}
\newcommand{\quat}{\mathbb{H}}
\newcommand{\Q}{\mathbb{Q}}

\newcommand{\spc}{\mathrm{spin}^c}

\newcommand{\conn}{{\mathscr A}}
\newcommand{\gauge}{{\mathscr G}}
\newcommand{\conf}{{\mathscr C}}
\newcommand{\quot}{{\mathscr B}}

\newcommand{\M}{{\mathcal M}}

\newcommand{\calM}{{\mathcal M}}

\newcommand{\calN}{{\mathcal N}}

\newcommand{\scrH}{{\mathscr H}}

\newcommand{\scrS}{{\mathscr S}}

\newcommand{\fraks}{\mathfrak{s}}
\newcommand{\frakt}{\mathfrak{t}}

\newcommand{\vp}{\varphi}

\newcommand{\del}{\partial}

\newcommand{\Ker}{\mathop{\mathrm{Ker}}\nolimits}
\newcommand{\Coker}{\mathop{\mathrm{Coker}}\nolimits}
\newcommand{\im}{\mathop{\mathrm{Im}}\nolimits}
\newcommand{\rank}{\mathop{\mathrm{rank}}\nolimits}
\newcommand{\Hom}{\mathop{\mathrm{Hom}}\nolimits}

\newcommand{\tilR}{\tilde{\mathbb{R}}}

\newcommand{\SW}{Seiberg--Witten }

\newcommand{\id}{\mathrm{id}}
\newcommand{\ind}{\mathop{\mathrm{ind}}\nolimits}

\title[PSC and 10/8 for end-periodic $4$-manifolds]{Positive scalar curvature and 10/8-type inequalities on $4$-manifolds with periodic ends}

\author{Hokuto Konno}
\address{2-1 Hirosawa, Wako, Saitama 351-0198, Japan}
\email{hokuto.konno@riken.jp}

\author{Masaki Taniguchi}
\address{Graduate School of Mathematical Sciences, the University of Tokyo, 3-8-1 Komaba, Meguro, Tokyo 153-8914, Japan}
\email{masakit@ms.u-tokyo.ac.jp}

\date{}

\begin{document}

\maketitle

\begin{abstract}
We show 10/8-type inequalities for some end-periodic 4-manifolds which have positive scalar curvature metrics on the ends. 
As an application, we construct a new family of closed 4-manifolds which do not admit positive scalar curvature metrics.
\end{abstract}

\tableofcontents

\section{Introduction}
In this paper, we shall give a relation between two different types of topics having independent histories via the \SW equations on some $4$-manifold with periodic ends.
The first topic is the existence of a metric with positive scalar curvature (PSC) on a given manifold.
This is a classical object of interest in Riemannian geometry.
The second is the 10/8-inequality,  regarded as one of the central topics in $4$-dimensional topology.
The relation exhibited in this paper between these two topics also yields a concrete application, that is, we construct a new family of closed 4-manifolds which do not admit positive scalar curvature metrics.

Let us start with the first topic above.
For a given manifold, the existence of a PSC metric is a fundamental problem in Riemannian geometry.
This problem was completely solved for simply connected closed $n$-manifolds with $n>4$~\cite{GL80,Sto92}.
In dimension $4$, the problem is still far from a satisfactory answer (even for simply connected manifolds), but there are two celebrated obstructions to PSC metric.
The first one is the vanishing of the signature $\sigma(X)$ of a closed oriented spin $4$-manifold $X$ under the assumption of the existence of a PSC metric.
The second is the vanishing of the Seiberg--Witten invariant of a closed oriented $4$-manifold $X$ with $b^{+}(X)>1$ under the same assumption, where $b^{+}(X)$ denotes the maximal dimension of positive definite subspaces of $H^{2}(X;\R)$ with respect to the intersection form.
Note that both of these two obstructions are valid only for $4$-manifolds having non-trivial second Betti numbers.
In contrast, we shall attack the problem $4$-manifolds having trivial second Betti numbers in this paper:
more precisely, we consider a closed oriented $4$-manifold $X$ such that $H_{\ast}(X;\Q) \cong H_{\ast}(S^1\times S^3;\Q)$. 
We call such $X$ a {\it rational homology $S^{1}\times S^{3}$}. We call a closed oriented $3$-manifold embedded into $X$ as a fixed generator of $H_3(X,\Z)$ a {\it cross-section of $X$}.
We also assume that $X$ contains an oriented rational homology $3$-sphere $Y$ as a cross-section.
For such a $4$-manifold, J.~Lin~\cite{L16} recently succeeded to construct the first effective obstruction to PSC metric using \SW theory on periodic-end $4$-manifolds.
(Although Lin originally considered an integral homology $S^{1}\times S^{3}$ in \cite{L16}, his result was generalized to any rational homology $S^{1} \times S^{3}$ in \cite{LRS17} by himself and D.~Ruberman and  N.~Saveliev.)
The remarkable obstruction due to Lin is described in terms of the Mrowka--Ruberman--Saveliev invariant $\lambda_{SW}(X,\fraks)$ defined in \cite{MRS11}, which depends on the choice of a spin structure $\fraks$ of $X$, and the Fr{\o}yshov invariant $h(Y, \frakt)$ defined in \cite{Fr10} for the restricted spin structure $\frakt = \fraks|_{Y}$ on $Y$ coming from $X$.
More precisely, Lin proved that, if $X$ admits a PSC metric, then the formula 
\begin{align}\label{obst}
\lambda_{SW}(X,\fraks)=-h(Y,\frakt)
\end{align}
holds.
By the use of this obstruction, Lin showed that any homology $S^1\times S^3$ which has $\Sigma(2,3,7)$ as a cross-section does not admit a PSC metric.

In this paper, we shall construct an obstruction which is different from Lin's one to PSC metric on homology $S^{1} \times S^{3}$.
To give the obstruction, we also consider \SW equations on periodic-end $4$-manifolds, used in Lin's argument.
However, our approach is based on a quite different point of view: the 10/8-inequality. 
Here we explain some historical background on the 10/8-inequality.
Given a non-degenerate symmetric bilinear form over $\Z$, it is quite natural to ask whether this is realized as the intersection form of a closed smooth $4$-manifold.
After the celebrated S.~Donaldson's diagonalization theorem~\cite{Do83}, the 
remaining problem centered on constraints on the intersection forms of  spin $4$-manifolds.
Y.~Matsumoto~\cite{Ma82} proposed a conjecture on such a constraint, called the 11/8-conjecture today.
If this conjecture would turn out true, the realization problem above shall be completely solved.
After the appearance of the Seiberg--Witten theory, M.~Furuta~\cite{Fu01} showed a strong constraint, now called the 10/8-inequality, on the intersection form of a smooth closed spin $4$-manifold.
For a long time, the 10/8-inequality has been the closest constraint to the 11/8-conjecture.
For this reason the 10/8-inequality is one of the central interests in $4$-dimensional topology.

Our main theorem, connecting PSC metrics with the 10/8-inequality, is described as follows:

\begin{theo}\label{main}
Let $(X,\fraks)$ be an oriented spin rational homology $S^1\times S^3$, $Y$ be an oriented rational homology $3$-sphere embedded in $X$, and $\frakt$ be the spin structure on $Y$ defined as the restriction of $\fraks$.
Suppose that $Y$ is a cross-section of $X$, i.e. $Y$ represents a fixed generator of $H_3(X;\Z)$.
Assume that $X$ admits a PSC metric.
Then, for any compact spin $4$-manifold $M$ bounded by $(Y,\frakt)$ as spin manifolds, the inequality
\[
b^+(M) \geq -\frac{\sigma (M)}{8}+h(Y,\frakt)
\]
holds.
Moreover, if $b^+(M)$ is an odd number, then we have
\[
b^+(M) \geq -\frac{\sigma (M)}{8}+h(Y,\frakt)+1,
\]
and if $b^+(M)$ is a positive even number, then we have
\[
b^+(M) \geq -\frac{\sigma (M)}{8}+h(Y,\frakt)+2.
\]
\end{theo}

\begin{rem}
Let $M'$ be a closed spin $4$-manifold and $M$ be the complement of an embedded $4$-disk in $M'$.
Since $S^1\times S^3$ has a PSC metric, we can substitute $X=S^1 \times S^3$ and $Y=S^3$ in \cref{main}.
Then the second or third inequality in \cref{main} recovers the original 10/8-inequality for $M'$ due to M.~Furuta (Theorem~1 in \cite{Fu01}).
\end{rem}

\Cref{main} is shown by considering the \SW equations on a periodic-end $4$-manifold, which is obtained by gluing $M$ with infinitely many copies of the compact $4$-manifold $W$ defined by cutting $X$ open along $Y$.
The inequalities in \cref{main} are derived as 10/8-type inequalities for this periodic-end spin $4$-manifold.
To show these 10/8-type inequalities, we use Y.~Kametani's argument~\cite{Ka18} which provides a 10/8-type inequality without using finite-dimensional approximations of the \SW equations.
On the other hand, D.~Veloso~\cite{Ve14} has considered the boundedness of finite-dimensional approximations of the \SW map on a periodic-end $4$-manifold under a similar assumption on PSC.
The authors expect that his argument may be also used to give similar 10/8-type inequalities.

\Cref{main} gives a new family of homology $S^1\times S^3$ which do not admit PSC metrics.
To describe our obstruction to PSC metric, it is convenient to use the following invariant.

\begin{defi}
\label{defi: epsilon}
For an oriented rational homology $3$-sphere $Y$ with a spin structure $\frakt$, we define a number $\epsilon (Y,\frakt) \in \Q$ by
\[
\epsilon(Y,\frakt)
:= \min \Set{ \frac{\sigma(M)}{8}+b^+(M) | 
\begin{matrix}
\text{$M$ is a compact spin $4$-manifold} \\
\text{bounded by $(Y,\frakt)$ as spin manifolds.}
\end{matrix}}.
\]
\end{defi}
A similar quantity is also used by C.~Manolescu~\cite{Ma14}.
Manolescu constructed an invariant $\kappa(Y,\frakt) \in \Q$ for a spin rational homology $3$-sphere $(Y,\frakt)$ and showed the inequality
\begin{align}\label{man}
\kappa(Y,\frakt)  \leq \epsilon(Y,\frakt)+1
\end{align}
in Theorem~1 of \cite{Ma14}. 
(For most of \cite{Ma14}, the results are stated for integral homology $3$-spheres.
See Remark~2 of \cite{Ma14} for rational homology $3$-spheres.)
Note that $\epsilon(Y,\frakt)$ is a well-defined finite number.
This is because the inequality \eqref{man} (and an analogous inequality for a rational homology $3$-sphere) provides a lower bound of $\epsilon(Y,\frakt)$, and every spin $3$-manifold bounds a compact spin $4$-manifold.
Using this invariant $\epsilon$, we define an invariant $\psi(Y,\frakt)$ of $(Y,\frakt)$ by
\begin{align}
\psi(Y,\frakt) :=- \epsilon(Y,\frakt) +h(Y,\frakt).
\label{eq: def of psi yt}
\end{align}
Using this map $\psi$, our obstruction to PSC metric is described as follows:

\begin{cor}
\label{cor psi examples}
Let $(Y,\frakt)$ be a spin rational homology $3$-sphere and suppose that $\psi(Y,\frakt)>0$.
Let $(X, \fraks)$ be a spin rational homology $S^1 \times S^3$.
If $X$ contains $ Y$ as a cross-section and $\fraks|Y= \frakt$ holds, then $X$ does not admit a PSC metric.
\end{cor}
\begin{proof}
For a given spin oriented rational homology $3$-sphere $(Y,\frakt)$,
let $M$ be a compact spin $4$-manifold with $\del M=Y$ as spin manifolds and $\epsilon(Y,\frakt)=\sigma(M)/8+b^+(M)$.
Let $(X,\fraks)$ be an oriented spin rational homology $S^1\times S^3$ which has $Y$ as a cross-section and suppose that $\fraks|_{Y}=\frakt$.
If $X$ admits a PSC metric, \cref{main} implies that
$\psi(Y,\frakt)=-b^+(M)-\sigma (M)/8 +h(Y,\frakt) \leq 0.$
This proves the \lcnamecref{cor psi examples}.
\end{proof}

Using \cref{cor psi examples}, we can construct many new examples of homology $S^1 \times S^3$'s which do not admit PSC metrics.
Such examples shall be given in \cref{examples}.

\begin{acknowledgement}
The authors would like to express their deep gratitude to Yukio Kametani for answering their many questions on his preprint~\cite{Ka18}.
The authors would also like to express their appreciation  to Mikio Furuta for informing them of Kametani's preprint and encouragements on this work.
The authors would like to express their deep gratitude to Danny Ruberman for giving comments on examples of this paper.
The authors also wish to thank Andrei Teleman for informing them of Veloso's argument~\cite{Ve14} and answering their  questions on it.
The authors would also like to express their appreciation to Jianfeng Lin and Fuquan Fang for pointing out the relation between our work and that of R.~Schoen and S.~T.~Yau~\cite{SY79}.
The authors also appreciate Ko Ohashi's, Mayuko Yamashita's, Kyungbae Park's and Kouki Sato's helpful comments on the paper~\cite{FK05}, on equivariant $KO$-theory, homology cobordisms and examples of homology $S^1\times S^3$'s respectively. 
The first author was supported by JSPS KAKENHI Grant Numbers 16J05569 and 19K23412.
The second author was supported by JSPS KAKENHI Grant Number 17J04364.
Both authors were supported by JSPS Grant-in-Aid for Scientific Research on Innovative Areas (Research in a proposed research area) No.17H06461 and the Program for Leading Graduate Schools, MEXT, Japan.
\end{acknowledgement}

\section{Preliminaries}
\label{section Preliminaries}
Let $(X,\fraks)$ be an oriented spin rational homology $S^1 \times S^3$ and $Y$ be an oriented rational homology $3$-sphere. We fix a Riemannian metric $g_X$ on $X$ and a generator of $H_3(X;\Z)$, denoted by $1\in H_3(X;\Z)$.
(Note that $H_3(X;\Z)$ is isomorphic to $H^{1}(X;\Z)$, and hence to $\Z$.)
We also assume that $Y$ is embedded into $X$ as a cross-section of $X$, namely $[Y]=1$.
Let $W_0$ be the rational homology cobordism from $Y$ to itself obtained by cutting $X$ open along $Y$.
The manifold $W_{0}$ is equipped with an orientation and a spin structure induced by that of $X$.
We define 
\[
W[m,n]:= W_m \cup_Y W_{m+1} \cup_Y \dots \cup_YW_n
\]
for $(m,n) \in (\{-\infty \}\cup \Z) \times (\Z\cup \{\infty\})$ with $m<n$.
Let us take a  compact spin $4$-manifold $M$ bounded by $Y$ as oriented manifolds.
The element $1 \in H^1(X;\Z)$ corresponding to $1 \in H_3(X;\Z)$ via Poincar\'e duality gives the isomorphism class of a $\Z$-bundle
\begin{align}\label{zcov}
p:\widetilde{X} \to X
\end{align}
and an identification 
\begin{align}\label{zcov1}
\widetilde{X} \cong W[-\infty,\infty].
\end{align}
 We can suppose that $H^1(M;\Z)=0$ by surgery preserving the intersection form of $M$ and the condition that $M$ is spin.

\begin{assum}\label{b1condition}
Henceforth we assume that $H^1(M;\Z)=0$.
 \end{assum}
 Then we get a non-compact manifold $Z:= M \cup_{Y} W[0,\infty]$ equipped with a natural spin structure induced by spin structures on $M$ and $W_0$. 
 Via the identification \eqref{zcov1}, we regard $p$ as a map from $W[-\infty ,\infty]$ to $X$. We set $p_+:W[0 ,\infty] \to X$ as the restriction of $p$.
 We call an object on $Z$ a {\it periodic object} on $Z$ if the restriction of the object to $W[0,\infty]$ can be identified with the pull-back of an object on $X$ by $p_+$.
For example, we shall use a periodic connection, a periodic metric, periodic bundles, and periodic differential operators.
By considering pull-back by $p_{+}$, the Riemannian metric $g_X$ on $X$ induces the Riemannian metric $g_{W[0,\infty]}$ on $W[0,\infty]$.
We extend the Riemannian metric $g_{W[0,\infty]}$ to a periodic Riemannian metric $g_Z$ on $Z$, and henceforth fix it.
Let $S^+, S^-$ be the positive and negative spinor bundles respectively over $Z$ determined by the metric and the spin structure.
If we fix a trivialization of the determinant line bundle of the spin structure on $Z$, we have the canonical reference connection $A_0$ on it corresponding to the trivial connection.

To consider the weighted Sobolev norms on $Z$, we fix a function  
\[
\tau : Z \to \R
\]
satisfying $T^*\tau =\tau+1$, where $T: W[0,\infty] \to W[0,\infty]$ is the deck transform determined by $T(W_i)=W_{i+1}$.

\subsection{Fredholm theory}
\label{delta0}
To obtain the Fredholm property of periodic elliptic operators on $Z$, it is reasonable to work on the $L^2_{k,\delta}$-norms rather than  the $L^2_k$-norms for fixed $k\geq3$ and a suitable weight $\delta$.
C.~Taubes \cite{T87} showed that a periodic elliptic operator on $Z$ with some condition is Fredholm with respect to $L^2_{k,\delta}$-norms for generic $\delta \in \R$.  
Let $\mathbb{D}= (D_i,E_i)$ be a periodic elliptic complex on $Z$, i.e. the complex 
 \begin{align} \label{comp}
0\to \Gamma(Z;E_N) \xrightarrow{D_N} \Gamma(Z;E_{N-1} )  \to \dots \xrightarrow{D_1} \Gamma(Z;E_0)\to 0 
\end{align}
satisfying
\begin{itemize} 
\item Each linear map $D_i$ is a first order periodic differential operator on $Z$.
\item  The symbol sequence of \eqref{comp} is exact.
\end{itemize} 
We consider the following norm 
 \[
 \| f \|_{L^2_{k,\delta}(Z)} := \| e^{\tau \delta} f \|_{L^2_{k}(Z)} 
 \]
 by using a periodic connection and a periodic metric. 
 We call the norm $\|-\|_{L^2_{k,\delta}}$ the {\it weighted Sobolev norm with weight $\delta \in \R$}.
By extending \eqref{comp} to the complex of the completions by the weighted Sobolev norms,  we obtain the complex of bounded operators 
 \begin{align}\label{elliop}
L^2_{k+N+1,\delta} (Z;E_N) \xrightarrow{D_N} L^2_{k+N,\delta}  (Z;E_{N-1} )  \to \dots \xrightarrow{D_1} L^2_{k,\delta} (Z;E_0) 
\end{align}
for each $\delta \in \R$.
Taubes constructed a sufficient condition for the Fredholm property of \eqref{elliop} by using the Fourier--Laplace (FL) transformation. 
The FL transformation replaces the Fredholm property of the periodic operator on $Z$ with the invertibility of a family of operators on $X$ parameterized by $S^1$.
Let us describe it below.
We first note that, since the operators in \eqref{elliop} are periodic differential operators, there are differential operators $\hat{\mathbb{D}}=(\hat{D}_i, \hat{E}_i)$ on $X$ such that there is an identification between $p_+^*\hat{\mathbb{D}}$ and $\mathbb{D}$ on $W[0,\infty]$.
The sufficient condition for Fredholmness is given by invertibility of the following complexes on $X$.
For $z \in \C$, we define the complex $\hat{\mathbb{D}}(z)$ by 
\begin{align} \label{comp2}
0 \to \Gamma(X;\hat{E}_N) \xrightarrow{\hat{D}_N(z)} \Gamma (X;\hat{E}_{N-1} )  \to \cdots \xrightarrow{\hat{D}_1(z)} \Gamma (X;\hat{E}_0) \to 0, 
\end{align} 
where the operator $\hat{D}_i(z): \Gamma(X;\hat{E}_i) \to \Gamma (X;\hat{E}_{i-1} )$ is give by
\[
\hat{D}_i(z)(f):= e^{-\tau z} \hat{D}_i (e^{\tau z}f).
\]

 \begin{theo}[Taubes, Lemma 4.3 and Lemma 4.5 in \cite{T87}]\label{fred}
 Suppose that there exists $z_0 \in \C$ such that the complex $\hat{\mathbb{D}}(z_0)$ is acyclic.
Then there exists a discrete subset $\mathcal{D}$ in $\R$ with no accumulation points such that $\eqref{elliop}$ is Fredholm for each $\delta$ in $\R \setminus \mathcal{D}$.
Moreover, the set $\mathcal{D}$ is given by 
\[
\mathcal{D} =\Set{ \delta \in \R |\hat{\mathbb{D}}(z) \text{ is not invertible for some $z$ with }  \text{Re } z=\delta .}.
\]
 \end{theo}
 \begin{rem} 
The assumption of \cref{fred} implies that the Euler characteristic of \eqref{comp2} is $0$ for all $z$.
We shall consider $\hat{\mathbb{D}}$ as the Atiyah--Hitchin--Singer complex, the spin (or spin$^c$) Dirac operator $D^+_A : \Gamma (X;S^+) \to  \Gamma  (X;S^-)$ or the de Rham complex. Note that the Euler characteristic (i.e. the index) of these operators are $0$ in our situation. 
\end{rem}
\begin{rem} The Fredholm property does not depend on the choice of $k$.
This is because the acyclic property of \eqref{comp2} does not depend on the choice of $k$ by the elliptic regularity theorem.
\end{rem}
If we consider the set $\mathcal{D}$ for the Atiyah--Hitchin--Singer complex on $Z$, one can show that the set $\mathcal{D}$ does not depend on the choice of Riemannian metric on $X$. However if we consider the spin (or spin$^c$) Dirac operator on $Z$, the set $\mathcal{D}$ depends on the choice of Riemannian metric.
Let us consider the following operator on $X$: 
\begin{align}\label{Dirac}
D^+_{A_0} + f^* d\theta : L^2_k (X;S^+) \to   L^2_{k-1} (X;S^-), 
\end{align}
where the map $f:X\to S^1$ is a smooth classifying map of \eqref{zcov}. 
 We call $g_X$ an {\it admissible metric} on $X$ if the kernel of \eqref{Dirac} is $0$.
 This condition is considered in \cite{RS07}.
The admissibility condition does not depend on the choice of classifying map $f$.

\begin{rem}
\label{rem: psc is admissible}
We can show that every PSC metric on $X$ is an admissible metric. This is a consequence of Weitzenb\"ock formula. 
(See (2) in \cite{RS07}.)
\end{rem}
Now we see that the assumption of \cref{fred} is satisfied for the operators in our situation.

\begin{lem}\label{allops}
The assumption of \cref{fred} is satisfied for the following operators: 
\begin{itemize}
\item The Dirac operator $D^+_{A_0}: L^2_{k,\delta} (Z;S^+) \to L^2_{k-1,\delta} (Z;S^-)$ for the pull-back of an admissible metric $g_X$ on $X$.
\item The Atiyah--Hitchin--Singer complex
\[
0 \to L^2_{k+1,\delta} (i\Lambda^0(Z)) \xrightarrow{d}  L^2_{k,\delta} (i\Lambda^1(Z)) \xrightarrow{d^+} L^2_{k-1,\delta} (i\Lambda^+(Z)) \to 0.
\]
\item The de Rham complex 
\begin{align} \label{derham}
0 \to L^2_{k+1,\delta} (i\Lambda^0(Z)) \xrightarrow{d}  L^2_{k,\delta} (i\Lambda^1(Z)) \xrightarrow{d} \dots \xrightarrow{d}  L^2_{k-3,\delta} (i\Lambda^4(Z)) \to 0.
\end{align}
\end{itemize}
\end{lem}

\begin{proof} 
The Fredholm property does not depend on the choice of $\tau$ satisfying $T^* \tau =\tau+1$ on $W[0,\infty]$.
Therefore we can choose a lift of $f$ as $\tau$. Then the operator $\hat{\mathbb{D}}(z_0)|_{z_{0}=1}$ corresponding to $D^+_{A_0}$ coincides with that corresponding to $D^+_{A_0} + f^* d\theta$.
Since the index of $D^+_{A_0} + f^* d\theta$ is $0$, admissibility implies that $\hat{\mathbb{D}}(z_0)|_{z_{0}=1}$ is acyclic.
The second condition follow from Lemma 3.2 in \cite{T87}. If $z=1$ and $D_i= d : \Omega^i(X) \to \Omega^{i+1}(X)$ , the operator $\hat{D}_N(z)$ can be described as follows: 
\[
\hat{D}_N(z) (f) = e^{- \tau}d( e^{\tau} f)= d f + d\tau \wedge f .
\]
 By the argument due to Taubes (Theorem 3.1 in \cite{T87}), the complex $( \Omega^i(X), f \mapsto  d f + d\tau \wedge f )$ is acyclic 
if and only if the following linear map gives a injective map: 
\[
\sigma_{i+1}: H_{\text{dR}}^i(X)/ \im \sigma_i \to H^{i+1}_{\text{dR}}(X), 
\]
where $H_{\text{dR}}^i(X)$ is the $i$-th de Rahm cohomology and $\sigma_{i+1} ([f]) := [d\tau \wedge f]$. 
The map $\sigma_0$,  $\sigma_1$, $\sigma_2$ and $\sigma_3$ are automatically injective since $d\tau$ ganerates $H^{1}_{\text{dR}}(X) \cong  \R$ and $H^{2}_{\text{dR}}(X) \cong  \{0\}$. By Poincar\'e duality, 
\[
\sigma_4 : H_{\text{dR}}^3(X) \to H^{4}_{\text{dR}}(X) 
\]
gives an isomorphism.  Therefore, $\sigma_4$ and $\sigma_5$ are also injective. This gives the conclusion.
\end{proof}

\begin{rem}
\label{rem deltazero}
Since $\mathcal{D}$ has no accumulation points, we can choose a sufficiently small $\delta_0>0$ satisfying that for any $\delta \in (0,\delta_0)$ the operators in \cref{allops} are Fredholm. We fix the notation $\delta_0$ in the rest of this paper.
\end{rem}

\subsection{Mrowka--Ruberman--Saveliev invariant and Lin's formula}
Let $(X,\fraks)$ be a spin rational homology $ S^1 \times S^3$.  
For such a $4$-manifold $X$, Mrowka--Ruberman--Saveliev~\cite{MRS11} constructed a gauge theoretic invariant $\lambda_{SW}$.
In this section, we review the definition of $\lambda_{SW}$ and the following result due to J.~Lin~\cite{L16}: the invariant $-\lambda_{SW}$ coincides with the Fr\o yshov invariant of its cross-section under the assumption that $X$ admit a PSC metric.

For a fixed spin structure, the formal dimension of the perturbed blown-up SW moduli space $\calM(X, g_X, \beta)$ of $X$ is 0.
Here $\beta$ denotes some perturbation.
Therefore the formal dimension of the boundary of  $\calM(X, g_X, \beta)$ is $ -1$.
Mrowka--Ruberman--Saveliev showed that the space  $\calM(X, g_X, \beta)$ has a structure of compact $0$-dimensional manifold for a fixed generic pair of a metric and a perturbation $(g_X, \beta)$. For a generic pair $(g_X, \beta)$, one can define the Fredholm index of the operator 
\[
D^+(Z, g_X, \beta) : L^2_{k}(Z;S^+) \to L^2_{k-1}(Z;S^-) .
\]
Note that, although $D^+(Z, g_X, \beta)$ is an operator over a manifold with periodic ends, we do not use weighted Sobolev norms to obtain the Fredholm property of $D^+(Z, g_X, \beta)$.
Instead, choosing a suitable pair $(g_X, \beta)$, one may ensure the Fredholm property under usual Sobolev norms.
Mrowka--Ruberman--Saveliev defined
\[
\lambda_{SW}(X,\fraks) := \# \calM(X, g_X, \beta) - \ind_{\C} D^+(Z, g_X, \beta) - \frac{\sigma(M)} {8}.
\]
in \cite{MRS11}.
Here $\#$ denotes the signed count of points in the moduli space.

Mrowka--Ruberman--Saveliev showed that $\lambda_{SW}(X,\fraks)$ does not depend on the choice of metric, perturbation, and $M$.
We also use the following theorem due to Lin~\cite{L16} and Lin--Ruberman--Saveliev~\cite{LRS17}. 
\begin{theo}[Lin, Theorem 1.2 in \cite{L16}, Lin--Ruberman--Saveliev, Theorem B in \cite{LRS17}]
\label{Lin}
Let $(X,\fraks)$ be an oriented spin rational homology $S^1\times S^3$ and $Y$ be an oriented rational homology $3$-sphere.
We fix a generator $H_3(X; \Z)$ and suppose  that $Y$ is embedded into $X$ as a submanifold such that $Y$ represents the fixed generator of $H_3(X; \mathbb{Z})$.
If $X$ has a PSC metric,  then the equality
\[
\lambda_{SW}(X,\fraks)=-h(Y,\frakt)
\]
holds.
\end{theo}

Recall that the Weitzenb\"ock formula implies that the SW moduli space is empty for a PSC metric.
(See \cite{M96}, for example.)
The following lemma is immediately deduced from this fact and the definition of $\lambda_{SW}(X,\fraks)$.

\begin{lem} \label{cor of W formula}
Let $(X,\fraks)$ be an oriented spin rational homology $S^1\times S^3$ and Y be an oriented rational homology $3$-sphere as in \cref{Lin}.
If $X$ admits a PSC metric $g_{X}$, the following equality holds: 
\[
\lambda_{SW}(X,\fraks) = -\ind_{\C} D^+(Z, g_X, \beta) - \frac{\sigma(M)} {8},
\]
where $M$ is a compact spin 4-manifold with $\partial M=Y$.

\end{lem}

\subsection{Kametani's theorem}
The original proof of the 10/8-inequality due to Furuta \cite{Fu01} for closed oriented spin 4-manifolds uses the properness property of the monopole map and a finite-dimensional approximation to that map.
After the work of Furuta, Bauer--Furuta \cite{BF04} constructed a cohomotopy version of the \SW invariant for closed oriented 4-manifolds by using the boundedness property of the monopole map and finite-dimensional approximation.
On the other hand, in \cite{Ka18},  Kametani developed a technique to obtain the 10/8-type inequality using only the compactness of the \SW moduli space.
In this section, we adapt Kametani's technique to obtain a 10/8-type inequality in our situation. First, we recall several definitions to formulate the theorem due to Kametani.

Let $G$ be a compact Lie group. 
\begin{defi}
Let $U$ be an oriented finite-dimensional vector space over $\R$ with an inner product.
A {\it real spin $G$-module} is a pair consisting of a representation $\rho : G \to SO(U)$ and a choice of lift $\tilde{\rho} : G \to Spin(U)$.
When a representation $\rho : G \to SO(U)$ is given, we call a lift $\tilde{\rho} : G \to Spin(U)$ a {\it lift to a real spin $G$-module} of $\rho$.
\end{defi}
\begin{rem}\label{real spin module}
Let $X$ be a $G$-space and $U$ be a real spin $G$-module. Suppose that the $G$-action on $X$ is free so that $(X\times U)/G \to X/G$ becomes a vector bundle. By the use of the structure of real spin $G$-module, one can show that $P_{X/G}:=(X\times Spin(n))/G \to X/G$ becomes a principal $Spin(n)$-bundle on $X/G$. The identification $P_{X/G} \times_{\pi} \R^n  \cong (X\times U)/G$ induces a spin structure on $(X\times U)/G \to X/G$, where $\pi$ is the double cover $Spin(n)\to SO(n)$.
\end{rem}

We consider the Lie group $Pin(2)$ which is the subgroup of $Sp(1) (\subset \quat)$ generated by $S^{1} (\subset \C \subset \quat)$ and $j \in \quat$. Let $\tilR$ be the non-trivial representation of $Pin(2)$ defined via the non-trivial homomorphism $Pin(2) \to \Z/2$ and the non-trivial real representation of $\Z/2$ on $\R$. 
We regard $\quat$ as the standard representation of $Pin(2)$ on the set of quaternions.
The real representation ring of $Pin(2)$ is given as follows:

\begin{lem}[See \cite{Lin15}, for example]\label{real representation of Pin(2)} The real representation ring $RO(Pin(2))$ of $Pin(2)$ can be described as follows: 
\begin{align} \label{relation}
RO( Pin(2) ) \cong \Z [ \tilR, \quat , \C]/ ( \tilR^2-1, \tilR \otimes \C - \C, \tilR \otimes \quat - \quat , \quat^2- 4 ( 1+ \tilR + \C )) , 
\end{align}
where $1$ corresponds to the one dimensional trivial representation, $S^1 \subset Pin(2)$ acts on $\C$ by $z \mapsto z^2$ and $j\in Pin(2)$ acts on $\C$ as the reflection along the diagonal. 
\end{lem}

\begin{lem} \label{spin G module H}
The $Pin(2)$-module $\quat$ has a lift to a real spin $Pin(2)$-module.
\end{lem}

\begin{proof}
Since the group $Spin(\quat ) \cong Sp(1) \times Sp(1)$ acts on $\quat$ by $v \mapsto \beta v \overline{\alpha}$ where $(\alpha, \beta) \in Sp(1) \times Sp(1) $,  the following diagram commutes:
\begin{align}
  \begin{CD}
    Pin(2) @>{}>> Spin(\quat ) \\
  @V{}VV    @VVV \\
     SO(\quat )  @ = SO(\quat ).
  \end{CD}
\end{align}
Here the upper horizontal map $Pin(2) \subset Sp(1)  \to   Sp(1) \times Sp(1) \cong Spin(\quat )$ is defined by $g\mapsto (1,g)$, and the left vertical map corresponds to the representation $\quat$ of $Pin(2)$.
This implies the conclusion.
\end{proof}

\begin{rem}
In this paper, for a fixed positive integer $m$, we equip $\quat^m$  with a structure of a real spin $Pin(2)$-module as the direct sum of the real spin $Pin(2)$-module defined in \cref{spin G module H}.
\end{rem}

Let $\Gamma$ be the pull-back of $Pin(1)$ along the map $Pin(2) \to O(1)$ (see Theorem 3.11 of \cite{ABS64}): 
\begin{align}\label{definition of Gamma}
  \begin{CD}
    \Gamma @>{}>> Pin(1) \\
  @V{}VV    @VVV \\
     Pin(2)   @>>>   O(1),
  \end{CD}
\end{align}
where  the map $Pin(2) \to  O(1)$ is the non-trivial homomorphism. 
The $\Gamma$-actions on $\quat$ and $\R$ are induced by $Pin(2)$-representations $\quat$ and $\tilR$ via \eqref{definition of Gamma}.
We denote these representations of $\Gamma$ by the same notations.
\begin{lem}\label{spin G module tilR}
For a positive number $n$ with $n \equiv 0 \mod 2$, $\tilR^{n}$ has a lift to a real spin $\Gamma$-module.
\end{lem}
\begin{proof} 
First, we consider the case that $n \equiv 0 \mod 4$ and put $n=4k$. Define $Pin(2) \to Spin(4)\cong Sp(1) \times Sp(1)$  by $g \mapsto (s(g), 1) \in Spin(4) $, where $s: Pin(2) \to \Z_2$ is the non-trivial homomorphism.
This map $Pin(2) \to Spin(4)$ covers the homomorphism $Pin(2) \to SO(4)$ corresponding to the representation $\tilde{\R}^{4}$.
Note that, for general $p,q \geq 1$, restricting a map between Clifford algebras, we obtain a natural map 
\[
Spin(p) \times Spin(q) \to Spin(p+q)
\]
covering the map $SO(p) \times SO(q) \to SO(p+q)$ defined by putting two matrices diagonally.
Therefore the above homomorphisms $Pin(2) \to Spin(4)$ and $Pin(2) \to SO(4)$ induces the diagram
\begin{align}
  \begin{CD}
    Pin(2) @>{}>> Spin(\R^{n} ) \\
  @V{}VV    @VVV \\
     SO(\R^{n})  @ = SO(\R^{n} )
  \end{CD}
\end{align}
via the direct products of $k$-copies of the homomorphisms.
The left vertical arrow is the map corresponding to $\tilde{\R}^{n}$, and thus we obtain a lift of $\tilde{\R}^{n}$ to a real spin $\Gamma$-module.

The remaining case is when $n$ is written as $n=4k+2$.
By the construction of $\Gamma$, there is the following commutative diagram: 
\begin{align}\label{diag}
  \begin{CD}
  \Z_2 @= \Z_2 @= \Z_2  \\
    @V{}VV    @VVV @VVV  \\
    \Gamma  @>>>   Pin(1)  @>>>           Spin(2 ) \\
  @V{}VV     @VVV  @VVV \\
    Pin(2)  @>>> O(1)  @>>> SO(2).
  \end{CD}
\end{align}
Taking the direct product of $(2k+1)$-copies of the map $\Gamma \to Spin(2)$ in \eqref{diag}, we obtain a homomorphism $\Gamma \to Spin(n)$ covering the map $\Gamma \to SO(n)$ corresponding to $\tilde{\R}^{n}$.
This proves the lemma.
\end{proof}
Now we show every real representation admits a lift to a real spin $Pin(2)$-module after considering stabilization. 

\begin{lem} \label{Pin(2)-module structure} For a given real  $Pin(2)$-representation $W$, there exists a  real  $Pin(2)$-representation $V$ such that $W\oplus V$ admits a lift to a real spin $Pin(2)$-module.
\end{lem}

\begin{proof}We can write $W$ as the following sum: 
\[
W = \sum_{ m(W)_{\tilR} , n(W)_\C, l(W)_\quat \geq 0 } m_{\tilR} \tilR \otimes n_\C \C \otimes  l_\quat \quat,
\]
where the sum means the direct sum of representations.
By considering the direct sum $W \oplus 7W$, we can assume that $m_{\tilR}$  is even and $n_\C \equiv 0 \mod 4$.  Using the last relation of \eqref{relation}, we can also assume that $8W$ has no $\C$ component. The representations $1$, $2\tilR$ and $\quat$ lift to real spin $Pin(2)$-modules by \cref{spin G module tilR} and \cref{spin G module H}. Since the tensor product of two real spin $G$-modules forms a real spin $G$-module in general, we have the conclusion.
\end{proof}

Let $V$ be a real spin $G$-module of dimension $n$. When $n \equiv 0 \mod 8$, there exists the Bott class $\beta (V) \in  KO^*(V)$ which generates the total cohomology ring $KO^*_G(V)$ as a $KO^*_G(pt)$-module due to the Bott periodicity theorem. For a general $n$, we fix a positive integer $m$ satisfying $m+n \equiv 0 \mod 8$ and define $\beta (V):= \beta(V \oplus \R^m) \in KO_G( V \oplus \R^m) \cong KO_G^n( V) $.
We define $e(V):= i^*\beta(V) \in KO^n_G(pt)$, where the map $i : pt \to V$ is the map defined by $i(pt)=0 \in V$. The class $e(V)$ is called the Euler class of $V$.

We use the notation $(P,\psi)$ for a spin structure on a manifold $M$.
It means that $\psi$ is a bundle isomorphism from $P \times_\pi \R^n$ to $TM$ as an $SO(n)$-bundle, where $\pi: Spin(n) \to SO(n)$ is the double cover and $P \times_\pi \R^n$ is the associated bundle for $\pi$.
\begin{defi} Let $M$ be a $G$-manifold of dimension $n$, $(P,\psi)$ be a spin structure on $M$ and $m: G\times P \to P$ be a $G$-action on the principal $Spin(n)$-bundle $P$ on $M$ which is a lift of the $G$-action on $M$.
The triple $(M,  (P,\psi) ,m )$ is called a {\it spin $G$-manifold} ($G$-manifold with an equivariant spin structure) if the following conditions are satisfied: 
\begin{enumerate}
 \item  The action $m$ commutes with the $Spin(n)$-action on $P$.
\item \label{0} Via $\psi$, the $G$-action on $P \times_\pi \R^n$ induced from the action on $P$ corresponds to the $G$-action on $TM$.
\end{enumerate} 
\end{defi}
\begin{rem}\label{equiv spin} Let $(M,  (P,\psi) ,m )$ be a spin $G$-manifold with free $G$-action.
Then we have the following diagram:  
\begin{align}
  \begin{CD}
    P @>{q^*}>> P/G \\
  @V{}VV    @VVV \\
    M  @>{q}>> M/G  ,
  \end{CD}
\end{align}
  where $q$ and $q^*$ are quotient maps.
Since the $G$-action is free on $M$, $M/G$ has a structure of a manifold. Since the $G$-action $m$ commutes with the $Spin(n)$ action on $P$, $Spin(n)$ acts on $P/G$. One can check that $P/G \to M/G$ determines a spin structure on $M/G$ by the second condition of the definition of spin $G$-manifold.
\end{rem}
\begin{rem} \label{equiv spin2}
Let $M$ be a $G$-manifold with free $G$-action. We also assume that $M/G$ has a spin structure.
We denote by $P_{ M/G}$ the principal  $Spin(n)$-bundle on $M/G$. Then we have the diagram: 
\begin{align}
  \begin{CD}
    q^* P_{ M/G} @>{}>> P_{ M/G} \\
  @V{}VV    @VVV \\
    M  @>{q}>> M/G  .
  \end{CD}
\end{align}
Since the quotient map $q: M \to M/G$ is a $G$-equivariant map (the $G$-action on $M/G$ is trivial), $q^* P_{ M/G}$ admit a $G$-action $m_{M/G}$ which commutes with $Spin(n)$-action. By the pull-back the identification $P_{ M/G} \times_\pi \R^n \cong T(M/G)$ by $q$, we obtain the identification $q^*P_{ M/G} \times_\pi \R^n \cong TM$. By the definition, one can check that $G$-action on $q^*P_{ M/G}$ and the $G$-action on $TM$ coincide. Therefore, $(M, q^*P_{ M/G}, m_{M/G})$ is a spin $G$-manifold. 

\end{rem}

We set 
\[
 \Omega^{\text{spin}}_{G, \text{free}}:= \{ \text{closed spin $G$-manifolds whose $G$-actions are free } \} / \sim.
 \] 
The relation $\sim$ is given as follows: $X_1 \sim X_2$ if there exists a compact spin $G$-manifold $Z$ whose $G$-action is free such that $\partial Z= X_1\cup (-X_2)$ as spin manifolds.
 For two real spin modules $U_0$, $U_1$ whose $G$-action on $U_0 \setminus \{ 0\}$ is free, we shall define an invariant $w(U_0,U_1)$ in $\Omega^{\text{spin}}_{G, \text{free}}$.
 To do this, we see that there exists a smooth $G$-map $S(U_0) \to U_1$ which is transverse to $0 \in U_1$ as follows.
Since the $G$-action on $S(U_{0})$ is free, the Borel construction
\[
S(U_{0}) \times_{G} U_{1} \to S(U_{0})/G
\]
gives us a vector bundle.
A section of this vector bundle which is transverse to the zero section corresponds to a $G$-map $S(U_0) \to U_1$ transverse to $0 \in U_{1}$.

\begin{defi} 
The element $w(U_0, U_1)  \in \Omega^{\text{spin}}_{G, \text{free}}$ is defined by taking a smooth $G$-map $\phi : S(U_0) \to U_1$ which is transverse to $0 \in U_1$ and setting $w(U_0, U_1):= [\phi^{-1} (0)]$. 
\end{defi}

Since $\Ker d \phi $ has the induced real spin $G$-module structure and the $G$-action on $\phi^{-1} (0)$ is free,  $w(U_0, U_1)$ determines the element in spin cobordism group with free $G$-action.
In \cite{Ka18}, it is shown that the class $w(U_0, U_1)$ is independent of the choice of $\phi$.

We use the following theorem due to Kametani.
\begin{theo}[Kametani \cite{Ka18}, Theorem~3] \label{Kametani}
Let $G$ be a compact Lie group.
Let $U_0$, $U_1$ be two real spin $G$-modules with $\dim U_0 = r_0$ and $\dim U_1 = r_1$.
Suppose that $G$-action is free on $U_0\setminus \{0\} $. If the cobordism class 
$w(U_0, U_1) \in \Omega^{\text{spin}}_{G, \text{free}}$ is zero, there exists an element $\alpha \in  KO_{G}^{r_1-r_0}(pt)$ such that 
\begin{align}\label{di}
e(U_1) = \alpha e(U_0).
\end{align}
\end{theo}

Furuta--Kametani \cite{FK05} showed the following inequality under the divisibility of the Euler class.
We shall combine \cref{Kametani} with \cref{Furuta Kametani} in \cref{sebsection: Completion of the proof of Main Thm}.

\begin{theo} [Furuta--Kametani \cite{FK05}]
\label{Furuta Kametani}
Let $ m_0$, $m_1$ be non-negative integers and
$l_1$ be a positive even number.
Suppose that there exists an element 
\[
\alpha \in  KO_{\Gamma}^{4m_1+l_1-4m_0}(pt)
\]
such that 
\begin{align}\label{divisibility}
e(\quat^{m_1})e( \tilR^{l_1}) = \alpha e(\quat^{m_0}) \in KO_{\Gamma}^{4m_1+l_1}(pt),
\end{align}
where the definition of $\Gamma$ is given in \eqref{definition of Gamma}.
Then the inequality 
\[
2(m_1-m_0)+ l_1-2 \geq 0
\]
holds.
\end{theo}

\begin{proof}
This \lcnamecref{Furuta Kametani} is deduced from Proposition~34 in \cite{FK05} as follows.
Although Proposition~34 is about an equivariant $KO$-theory of an $n$-dimensional torus $\tilde{T}^{n}$ with some group action, for our purpose, we need only the case that $n=0$, namely an equivariant $KO$-theory of a point.
Let $S=\emptyset$ in the setting of Proposition~34.
In this situation, we may see that $(\alpha_{\varphi})_{S} = \alpha_{\varphi}$ (see Lemma~31 in \cite{FK05}).
Moreover, we have $N_{S}=1$ for $S=\emptyset$.
(See the sentence between Theorem~2 and Corollary~3 in \cite{FK05}.)
Here some notation in this paper corresponds to that in \cite{FK05} as follows:
\begin{align}
\label{eq: correspondence between KT and FK}
m_{0}=y+k,\quad m_{1}=y\quad {\rm and}\quad  l_{1}=l.
\end{align}

We first check the divisibility condition \eqref{divisibility} implies (14) in \cite{FK05} provided that $n=0$ and $S=\emptyset$.
The notation $\quat_{1}$ in \cite{FK05} is the representation $\quat$ in this paper.
The number $A$ and $A_{1}$ are zero in the case that $S = \emptyset$.
Moreover, the last two factors in the right-hand-side of (14) in \cite{FK05} are equal to $1$ again for $S = \emptyset$.
Therefore it follows from \eqref{eq: correspondence between KT and FK} that the condition \eqref{divisibility} is equivalent to (14) in \cite{FK05} for $n=0$ and $S=\emptyset$.

Second we check that the desired inequality for $l_{1}, m_{0}$ and $m_{1}$ follows from Proposition~34.
Because of \eqref{eq: correspondence between KT and FK}, the number $d=l-4k$ in \cite{FK05} corresponds to $l_{1}-4(m_{0}-m_{1})$.
Since we supposed that $l_{1}$ is an even number, 
so is $l_{1}-4(m_{0}-m_{1})$.
For such $d$, 
Proposition~34 implies that $l/2-k-1 \geq 0$, namely $l_{1}/2 + m_{1}-m_{0} -1 \geq 0$.
This is the desired inequality.
\end{proof}

\begin{rem}
A sketch of the proof of the Proposition~34 in \cite{FK05} is as follows.
Through a direct computation,
Proposition~34 follows from Lemma~35 in \cite{FK05}, which gives a presentation of an element $\alpha \in KO_{\Gamma}^{\rm even}(\tilde{\R}^{\rm even})$ appearing in an equation involving the Euler classes for the representations $\quat$ and $\tilde{\R}$.
One may calculate the $j$-trace and the $j^{2}$-trace for the image of $\alpha$ under the complexification $KO_{\Gamma}^{\rm even} \to K_{\Gamma}^{\rm even}$, and this calculation determines the complexification of $\alpha$.
The kernel of the complexification is shown to be torsion in Lemma~15 in \cite{FK05}, and this is enough to prove the statement of Lemma~35.
\end{rem}

\subsection{Moduli theory}
In this subsection, we review the moduli theory for 4-manifolds with periodic ends. The setting of gauge theory for such manifolds is developed by Taubes in \cite{T87}.  
All functional spaces appearing in this subsection are considered on the end-periodic $4$-manifold $Z$, introduced at the beginning of this section, and therefore we sometimes drop $Z$ from our notation.

We fix a real number $\delta$ satisfying $0<\delta<\delta_0$ and an integer $k \geq 3$, where $\delta_0$ is introduced in Subsection 2.1.
The space of connections of the determinant line bundle of the given spin structure is defined by 
 $\conn_{k,\delta}(Z) := A_0 + L^2_{k,\delta}(i\Lambda^1(Z))$. We set the configuration space by 
 $\conf_{k,\delta}(Z)  :=\conn_{k,\delta}(Z) \oplus L^2_{k,\delta}( S^+)$.   The irreducible part of  $\conf_{k,\delta}(Z)$ is denoted by $\conf^*_{k,\delta}(Z)$.
 The gauge group $\gauge_{k+1,\delta}$ for the given spin structure is defined by 
 \[
 \gauge(Z)_{k+1,\delta}:= \Set{ g \in L^2_{k+1,\text{loc}} (Z, S^1) | dg \in L^2_{k,\delta} }.
 \]
 The topology of  $\gauge(Z)_{k+1,\delta}$ is given by the  metric
 \[
 \|g-h\| := \|dg - dh\|_{L^2_{k,\delta}} + |g(x_0)-h(x_0)| ,
 \]
 where $x_0 \in W_0$ is a fixed point.
The space $\gauge_{k+1,\delta}$ has a structure of a Banach Lie group.
Let us define a normal subgroup of  $\gauge(Z)_{k+1,\delta}$ (corresponding to the so-called based gauge group) by
\[
\widetilde{\gauge}(Z)_{k+1,\delta}:= \Set{ g \in  \gauge(Z)_{k+1,\delta} | L_{x_0}(g) =1 },
\]
where $L_{x_0}(g) = \displaystyle \lim_{n \to \infty} g(T^n(x_0))$.
Note that we have the exact sequence 
\[
1 \to \widetilde{\gauge}(Z)_{k+1,\delta} \to  \gauge(Z)_{k+1,\delta} \xrightarrow{L_{x_0}} S^1 \to 1.
\]
 The space $\conf_{k,\delta}(Z)$ is acted by $\gauge_{k+1,\delta}$ via pull-back, and moreover
one can show that ${\gauge}_{k+1,\delta}$ acts smoothly on $\conf_{k,\delta}(Z)$ and  $\widetilde{\gauge}_{k+1,\delta}$ acts freely on $\conf_{k,\delta}(Z)$. 
The tangent spaces of $\gauge_{k+1,\delta}$ and  $\widetilde{\gauge}_{k+1,\delta}$ can be described as follows.
(See Lemma~7.2 in \cite{T87})

\begin{lem}The following equalities
\begin{align*}
T_e\widetilde{\gauge}(Z)_{k+1,\delta} 
& = \Set{a \in L^2_{k+1,\text{loc}} (i\Lambda^1(Z))| da \in L^2_{k,\delta},\  \lim_{n\to \infty} a(T^n(x_0))= 0}\\
& = L^2_{k+1,\delta}(i \Lambda (Z) ), \text{and}\\
T_e {\gauge}(Z)_{k+1,\delta}
& = \Set{a \in L^2_{k+1,\text{loc}} (i\Lambda^1(Z))| da \in L^2_{k,\delta}}
\end{align*}
hold.
\end{lem}

We use the following notations:  
\begin{itemize}
\item $\quot_{k,\delta}(Z):= \conf_{k,\delta}(Z)/ \gauge(Z)_{k+1,\delta}$, 
\item $\widetilde{\quot}_{k,\delta}(Z):= \conf_{k,\delta}(Z)/ \widetilde{\gauge}(Z)_{k+1,\delta}$ and
\item  $\quot^*_{k,\delta}(Z):= \conf_{k,\delta}^*(Z)/ \gauge(Z)_{k+1,\delta}$.
\end{itemize}

As in Lemma~7.3 of \cite{T87}, one can show that the spaces ${\quot}^*_{k,\delta}(Z)$ and $\widetilde{\quot}_{k,\delta}(Z)$ have structures of Banach manifolds.
In the proof of this fact, the following decomposition is used. 

\begin{lem}\label{lem: Lambda one decomposition}For a fixed real number $\delta$ with $0<\delta < \delta_0$, there is the following $L^{2}_{\delta}$-orthogonal decomposition
\begin{align*}
L^{2}_{k,\delta}(i\Lambda^{1}(Z))
=& \Ker(d^{\ast_{L^2_\delta}} : L^{2}_{k,\delta}(i\Lambda^{1}(Z)) \to L^{2}_{k-1,\delta}(i\Lambda^{0}(Z)))\\
 &\oplus
\im(d : L^{2}_{k+1,\delta}(i\Lambda^{0}(Z)) \to L^{2}_{k,\delta}(i\Lambda^{1}(Z))).
\end{align*}

\end{lem}

\begin{proof}
Since the operator \eqref{derham} is Fredholm by the choice of $\delta_0$ (see the end of \cref{delta0}), the proof is essentially same as in the case of the decomposition for closed oriented $4$-manifolds. 
\end{proof}

The {\it monopole map}  $\nu_{h} : \conf_{k,\delta}(Z) \to L^2_{k-1,\delta}(i\Lambda^+(Z)\oplus S^-)$ is defined by 
 \[
 \nu_h(A,\Phi):= (F_A^+- \sigma(\Phi,\Phi)- i h , D_A(\Phi)),
 \]
 where $\sigma(\Phi,\Phi)$ is the trace-free part of $\Phi \otimes \Phi^*$ and regarded as an element of $L^{2}_{k-1,\delta}(i\Lambda^{+}(Z))$ via the Clifford multiplication and $h$ is a compactly supported self-dual $2$-form.
 We denote  $\nu_0(A,\Phi)$ by  $\nu_h(A,\Phi)$.
Recall that the map $L^2_{k,\delta} \times L^2_{k,\delta} \to L^2_{k ,\delta}$ is continuous for $k>2$ because of the Sobolev multiplication theorem.
Since we consider a spin structure rather than general $\spc$ structures, the monopole map is a $Pin(2)$-equivariant map.
 We define the {\it monopole moduli spaces} for $Z$ by
  \[
\M_{k,\delta,h}(Z):= \Set{ [(A, \Phi)] \in \quot_{k,\delta}(Z) | \nu_h (A,\Phi) =0 }, 
\]
 \[
\M^*_{k,\delta,h}(Z):= \Set{ [(A, \Phi)] \in \quot^*_{k,\delta}(Z) | \nu_h (A,\Phi) =0 }\text{ and }
\]
 \[
\widetilde{\M}_{k,\delta,h}(Z):= \Set{ [(A, \Phi)] \in \widetilde{\quot}_{k,\delta}(Z) | \nu_h (A,\Phi) =0 }.
\]
For simplicity, we denote $\M_{k,\delta,0}(Z)$, $\M^*_{k,\delta,0}(Z)$ and $ \widetilde{\M}_{k,\delta,0}(Z)$ by $ \M_{k,\delta}(Z)$, $\M^*_{k,\delta}(Z)$ and $\widetilde{\M}_{k,\delta}(Z)$.
At the end of this subsection, we study the local structure of $d\nu$ near $[(A_0,0)]$.
We consider the following bounded linear map 
\begin{align}\label{reducible operator}
 (d \nu+ d^{*_{L^2_\delta}})_{(A_0,0)} : \conf_{k,\delta}(Z)   \to L^2_{k-1,\delta}( S^-\oplus i \Lambda^+ \oplus  i\Lambda^0).
 \end{align}

\begin{prop}
\label{formal dimension}
Suppose that $X$ admits a PSC metric.
Then there exists $\delta_1 \in (0, \delta_0)$ satisfying the following condition:
for each $\delta \in (0, \delta_1)$, there exist positive numbers $l_0$ and $l_1$ with $l_1-l_0 =2 \ind_{\C} (D_{A_0}:L^2_k(Z;S^{+}) \to L^2_{k-1}(Z;S^{-}))$ such that there exist isomorphims
\begin{itemize}
\item 
 $\Ker  (d \nu+ d^{*_{L^2_\delta}})_{(A_0,0)} \cong  \quat^{l_1}$,  
 \item   $\Coker  (d \nu+ d^{*_{L^2_\delta}})_{(A_0,0)} \cong  \quat^{l_0} \oplus \tilR^{b^+(M)}$ 
 \end{itemize}
as representations of $Pin(2)$.
\end{prop}

\begin{proof}
It is easy to show that the operator \eqref{reducible operator} is the direct sum of 
\begin{align}\label{op1}
d^{+}+ d^{*_{L^2_\delta}} :L^2_{k,\delta}(  i \Lambda^1 )\to L^2_{k-1,\delta}(  i \Lambda^+ \oplus  i\Lambda^0)
\end{align}
and 
\begin{align}\label{op2}
D^+_{A_0}: L^2_{k,\delta}( S^+)\to L^2_{k-1,\delta}( S^-).
\end{align}
Taubes (Proposition 5.1 in \cite{T87}) showed that  
the kernel of \eqref{op1} is isomorphic to $\R^{b_1(M)}$ and  the cokernel of \eqref{op2} is isomorphic to $\R^{b^+(M)}$ for small $\delta$.
On the other hand, since $g_X$ is a PSC metric, the operator
\[
D^+(Z,g_X,0):L^2_{k, \delta} (S^+) \to L^2_{k-1, \delta}(S^-)
\]
is Fredholm for any $\delta$. (Use (2) in \cite{RS07}).
This implies that 
\[
\ind_{\C} (D_{A_0}:L^2_k(Z) \to L^2_{k-1}(Z))=\ind_{\C} (D_{A_0}:L^2_{k,\delta}(Z) \to L^2_{k-1,\delta}(Z))
\]
for any $\delta$.
Therefore, using \cref{b1condition}, we can see that
\[
\Ker  (d \nu+ d^{*_{L^2_\delta}})_{(A_0,0)} \cong  \quat^{l_1}
\]
and
\[
\Coker  (d \nu+ d^{*_{L^2_\delta}})_{(A_0,0)} \cong \quat^{l_0} \oplus \R^{b^+(M)}
\]
as vector spaces for some $l_{0}, l_{1}$ with $l_1-l_0 =2 \ind_{\C} (D_{A_0}:L^2_k(Z) \to L^2_{k-1}(Z)$).
Since $d^{+}+ d^{*_{L^2_\delta}} $ is a $Pin(2)$-equivariant linear map, its kernel and cokernel have structures of $Pin(2)$-modules and these representations are the direct sum of  $\quat$ and $\tilR$.
\end{proof}

\begin{prop} \label{reducible}Suppose that $b^+(M)=0$.
For each $\delta \in (0, \delta_1)$ where $\delta_1$ is the positive constant in \cref{formal dimension}, the space 
$\M_{k,\delta,h}(Z) \setminus \M^*_{k,\delta,h}(Z)$ contains just one point for any $h$.
 \end{prop}
\begin{proof} 
Since the map $L^2_{k,\delta }( \Lambda^1 (Z)) \xrightarrow{d^+} L^2_{k,\delta }( \Lambda^+ (Z))$ is surjective due to the calculation of Taubes (Proposition 5.1 in \cite{T87}) and the condition $b^+(M)=0$, we have a solution $A_0$ to the equation 
\[
F^+(A_0)= h .
\]
We have the corresponding element $(A_0, 0)$ in the configuration space and this gives the existence of reducible solution.

Suppose we have elements $[(A_i,0)] \in \M_{k,\delta}(Z) \setminus \M^*_{k,\delta}(Z)$ for $i=0$ and $i=1$. 
The connections $A_0$ and $A_1$  determine elements in $H^1(d, L^2_{k,\delta} )$, where the group $H^1(d, L^2_{k,\delta} )$  is the first cohomology of the following complex: 
\[
0 \to L^2_{k+1,\delta }( \Lambda^0 (Z))  \xrightarrow{d} L^2_{k,\delta }( \Lambda^1 (Z)) \xrightarrow{d^+} L^2_{k,\delta }( \Lambda^+ (Z))\to 0 .
\]
The result of Taubes(Proposition 5.1 in \cite{T87}) implies that $H^1(d, L^2_{k,\delta}) =0$ if $\delta \in (0, \delta_1)$ . Therefore the classes satisfy  $[A_0 ] = [A_1] $.
 Then we have a $L^2_{k+1,\delta}$-function $w$ such that 
 \[
-i A_0 = dw -i  A_1
 \]
 If we put $g := i e^w$ , then we have $A_1= g^*A_0$.  This gives the conclusion.

\end{proof}

\section{Proof of \cref{main}}
In this section, we give the proof of \cref{main}.
We first consider a combination of the Kuranishi model and some $Pin(2)$-equivariant perturbation, obtained by using some arguments of Y.~Ruan's virtual neighborhood technique~\cite{Ru98}.
Using it, we shall show a divisibility theorem of the Euler class following Y.~Kametani~\cite{Ka18}.
This argument produces the 10/8-type inequality on periodic-end spin $4$-manifolds.

\subsection{Perturbation}

Let $(X,\fraks)$ be the an oriented spin rational homology $S^{1} \times S^{3}$ and $Z$ be the periodic-end $4$-manifold given at the beginning of \cref{section Preliminaries}.
Suppose that $X$ admits a PSC metric $g_X$.
Then \cref{fred}, \cref{rem: psc is admissible}, and \cref{allops} imply that the Dirac operator over $Z$ is a Fredholm operator for the pull-back metric and a suitable weight $\delta$.

We confirm that what is called the global slice theorem holds also for our situation.
Henceforth we use this notation $d^{\ast}$ for the formal adjoint of $d$ with respect to the $L^{2}_{\delta}$-norm if no confusion can arise.
Let us define
\[
\scrS_{k,\delta} := \Ker(d^{\ast} : L^{2}_{k,\delta}(i\Lambda^{1}(Z)) \to L^{2}_{k-1,\delta}(i\Lambda^{0}(Z))) \times L^{2}_{k,\delta}(S^{+}).
\]

\begin{lem}
\label{lem: global slice}
The map
\[
\scrS_{k,\delta} \times \widetilde{\gauge}(Z)_{k+1,\delta} \to \conf_{k,\delta}(Z)
\]
defined by 
\[
((a, \Phi),g) \mapsto g^{\ast}(A_{0}+a, \Phi)
\]
is a $\widetilde{\gauge}(Z)_{k+1,\delta}$-equivariant diffeomorphism.
In particular, we have
\[
\scrS_{k,\delta} \cong \widetilde{\quot}_{k,\delta}(Z).
\]
\end{lem}

\begin{proof}
The assertion on $\widetilde{\gauge}(Z)_{k+1,\delta}$-equivariance is obvious.
To prove that the map given in the statement is a diffeomorphism, it suffices to show that the map
\[
\vp : \Ker(d^{\ast} : L^{2}_{k,\delta}(i\Lambda^{1}(Z)) \to L^{2}_{k-1,\delta}(i\Lambda^{0}(Z))) \times \widetilde{\gauge}(Z)_{k+1,\delta} \to L^{2}_{k,\delta}(i\Lambda^{1}(Z))
\]
defined by $(a,g) \mapsto a-2g^{-1}dg$ is a diffeomorphism.
Henceforth we simply denote by $\Ker{d^{\ast}}$ the first factor of the domain of this map if there is no risk of confusion.

We first show that $\vp$ is surjective.
Take any $a \in \Ker{d^{\ast}}$.
Thanks to the $L^{2}_{\delta}$-orthogonal decomposition given in \cref{lem: Lambda one decomposition},
we can find $f \in L^{2}_{k+1,\delta}(i\Lambda^{0}(Z))$ such that $-2df = a-p(a)$, where $p$ is the $L^{2}_{\delta}$-orthogonal projection to $\Ker{d^{\ast}}$ from $L^{2}_{k,\delta}(i\Lambda^{1}(Z))$.
Set $g := e^{f}$.
Since $f$ decays at infinity, $g \in \widetilde{\gauge}(Z)_{k+1,\delta}$ holds, and we get $\vp(p(a),g) = a$.

We next show that $\vp$ is injective.
Assume that $\vp(a,g) = \vp(a',g')$ holds for $(a,g), (a',g') \in \Ker{d^{\ast}} \times \widetilde{\gauge}(Z)_{k+1,\delta}$.
Then we have $a-a' - 2(gg')^{-1}d(gg')=0$.
Therefore, to prove that $\vp$ is injective, it suffices to show that, for 
$(a,g) \in \Ker{d^{\ast}} \times \widetilde{\gauge}(Z)_{k+1,\delta}$, if $\vp(a,g)=0$ holds we have $a=0$ and $g=1$.
Assume that $\vp(a,g)=0$.
Then, since $a \in \Ker{d^{\ast}}$, we have $d^{\ast}(g^{-1}dg)=0$.
On the other hand, $d(g^{-1}dg)=0$ also holds, and thus we can use the elliptic regularity.
Therefore $g^{-1}dg$ has the regularity of $C^{\infty}$.
Since $b_{1}(Z)=0$, there exists a function $h \in C^{\infty}(i\Lambda^{0}(Z))$ such that $dh = g^{-1}dg$.
By the argument after Lemma~5.2 of Taubes~\cite{T87}, we can take $h$ to be $L^{2}_{\delta}$.
Since $g^{-1}dg \in L^{2}_{k,\delta}$ holds, we finally get $h \in L^{2}_{k+1,\delta}$.
Since $h$ decays at infinity, one can integrate by parts: $0 = (d^{\ast}dh, h)_{L^{2}_{\delta}} = \|dh\|_{L^{2}_{\delta}}^{2}$, and hence $dh=0$.
Therefore $h$ is constant, and moreover, $h$ is constantly zero again because of the decay of $h$.
Thus we get $g^{-1}dg=0$, and hence $g$ is constant and $a=0$.
Since $\lim_{n \to \infty}g(T^{n}(x_{0}))=1$, we finally have $g=1$.
This completes the proof.
\end{proof}

Using \cref{lem: global slice} and restricting the map $\nu : \conf_{k,\delta}(Z) \to L^2_{k-1,\delta}(\Lambda^+(Z) \oplus S^-)$ corresponding to the \SW equations to the global slice, we get a map from $\scrS_{k,\delta}$, denoted by $\mu$:
\begin{align} \label{definition of mu}
\mu : \scrS_{k,\delta} \to L^2_{k-1,\delta}(\Lambda^+(Z) \oplus S^-).
\end{align}
This is a $Pin(2)$-equivariant non-linear Fredholm map.

\begin{rem}\label{invariant norm}
Note that, although the $L^2_{\delta}$-norm is $Pin(2)$-invariant, the $L^2_{k,\delta}$-norm is not $Pin(2)$-invariant in general.
However, by considering the average with respect to the $Pin(2)$-action, one can find a $Pin(2)$-invariant norm which is equivalent to the usual $L^{2}_{k,\delta}$-norm induced by the periodic metric and periodic connection.
Henceforth we fix this $Pin(2)$-invariant norm, and just call it a $Pin(2)$-invariant $L^{2}_{k,\delta}$-norm and denote it by $\|\cdot\|_{L^{2}_{k,\delta}}$.
\end{rem}

Via the isomorphism given in \cref{lem: global slice}, the quotient $\mu^{-1}(0)/S^{1}$ can be identified with the moduli space $\M_{k,\delta}(Z)$, and thus we get the following result by using the technique in Lin~\cite{L16}.

\begin{prop}
\label{cpt} There exists $\delta_2> 0$ satisfying the following condition.
For any $\delta \in (0,\delta_2)$, the space $\mu^{-1}(0)/S^{1}$ is compact and the space $\widetilde{\M}_{k,\delta,h}(Z)$ is also compact.
\end{prop}

\begin{proof}
By using the identification between $\mu^{-1}(0)/S^{1}$ and $\M_{k,\delta}(Z)$, it is sufficient to show $\M_{k,\delta}(Z)$ is compact. The proof of the second claim is similar to the first one. 
 Let $\{[(A_n, \Phi_n)]\}  \subset \M_{k,\delta}(Z)$ be any sequence in $\M_{k,\delta}(Z)$. Since $(A_n, \Phi_n)$ converges $(A_0,0)$ on the end for each $n$, the topological energy
\[
\mathcal{E}^{\text{top}}(A_n, \Phi_n):= \frac{1}{4}\int_{Z} F_{A^t_n} \wedge F_{A^t_n}
\]
defined in the book of Kronheimer--Mrowka~\cite{KM07} has a uniform bound 
\[
\mathcal{E}^{\text{top}}(A_n, \Phi_n) \leq C.
\]
We set $W(\epsilon, \infty]:= W[0,\infty] \setminus Y \times [0,\epsilon]$ where $Y \times [0,\epsilon]$ is a closed color neighborhood of $Y \times 0 \subset W[0,\infty]$.
The uniform boundedness of $\mathcal{E}^{\text{top}}(A_n, \Phi_n)$ and Theorem~4.7 in \cite{L16 } imply that $\{[(A_n,\Phi_n)|_{W[\epsilon, \infty]}]\}_{n}$ has a convergent subsequence in the $L^2_{k,\delta}$-topology.
(In Theorem~4.7 in \cite{L16}, Lin imposed the boundedness of $\Lambda_q$.
This is because Lin considered the blown-up moduli space.
On the other hand, for the convergence in the un-blown-up moduli space, we only need the boundedness of the energy.)
Therefore we have gauge transformations $\{g_n\}$ over $W(\epsilon, \infty]$ such that $\{(g_n^*A_n, g_n^*\Phi_n)\}$ converges in $L^2_{k,\delta}(W(\epsilon, \infty]; i\Lambda^1 \oplus S^+)$. On the other hand, we also have an energy bound 
\[
\mathcal{E}^{\text{top}}((A_n,\Phi_n)|_{M\cup_Y W_0}) \leq C .
\]
By Theorem 5.1.1 in \cite{KM07}, we have gauge transformation $h_n$ on $M\cup_Y Y \times [0,\epsilon']$ such that $\{(h_n^*A_n,h_n^*\Phi_n)\}$ has a convergent subsequence in $L^2_{k}(M\cup_Y Y \times [0,\epsilon'];i\Lambda^1 \oplus S^+)$ for $\epsilon <\epsilon'$.
Pasting $g_n$ and $h_n$ via a bump-function, we get gauge transformations $\{g_n\# h_n\}$ defined on the whole of $Z$ satisfying that 
$\{(g_n\# h_n^*A_n, g_n\# h_n \Phi_n )\}$ has convergent subsequence in $L^2_{k,\delta}(Z; i\Lambda^1 \oplus S^+)$. (This is a standard patching argument for gauge transformations. For example, see Sub-subsection~4.4.2 in \cite{DK90}.)
\end{proof}

Set
\[
\scrH_{k-1,\delta} := L^2_{k-1,\delta}(\Lambda^+(Z)) \times L^2_{k-1,\delta}(S^{-}).
\]
For a positive real number $\eta$, we define
\[
\mathbb{B}(\eta):=  \set {x \in  \scrS_{k,\delta} |  \| x \|^2_{L^2_{k,\delta}} <\eta }
\]
Since our $L^2_{k,\delta}$-norm is $Pin(2)$-invariant (see \cref{invariant norm}), $Pin(2)$ acts  on $\scrS_{k,\delta}\setminus \mathbb{B}(\eta)$.
Therefore the space $\scrS_{k,\delta}\setminus \mathbb{B}(\eta)$ has a structure of $Pin(2)$-Hilbert manifold with boundary.
Here let us recall that we introduced positive numbers $\delta_{1}$ and $\delta_{2}$ in \cref{formal dimension} and in \cref{cpt} respectively.
The following \lcnamecref{per} ensures that we can take a suitable and controllable perturbation of the \SW equations outside a neighborhood of the reducible.

\begin{lem}
\label{per}
For any $\delta >0 $ with $0< \delta < \min (\delta_1, \delta_2) $, $\eta>0$ and $\epsilon>0$, 
there exists a $Pin(2)$-equivariant smooth map
\[
g_\epsilon : \scrS_{k,\delta}\setminus \mathbb{B}(\eta) \to \scrH_{k-1,\delta}
\]
satisfying the following  conditions: 
\begin{enumerate} 
 \item For every point $\gamma \in  (\mu|_{\scrS_{k,\delta}\setminus \mathbb{B}(\eta)}+g_\epsilon)^{-1}(0)$, the differential
 \[
d(\mu+g_\epsilon)_{\gamma} :  \scrS_{k,\delta}    \to \scrH_{k-1,\delta}
 \]
 is surjective.

 \item Any element of the image of $g_\epsilon$ is smooth.
 \item There exists $N>0$, depending on $\epsilon$, such that 
 \[
 g_\epsilon(x)|_{W[N,\infty]}=0
 \]
 holds for any $x \in \scrS_{k,\delta}\setminus \mathbb{B}(\eta)$.
 \item There exists a constant $C>0$ independent of $\epsilon$ such that
 \[
 \|(dg_\epsilon)_{x}: \scrS_{k,\delta} \to \scrH_{k-1,\delta}\|_{\mathcal{B}(\scrS_{k,\delta}, \scrH_{k-1,\delta})}< C\epsilon
 \]
holds for any $x \in \scrS_{k,\delta}\setminus \mathbb{B}(\eta)$.
Here $\|\cdot\|_{\mathcal{B}(\cdot, \cdot)}$ denotes the operator norm.
 \item There exists a constant $C'>0$ independent of $\epsilon$ such that 
 \[
\|g_\epsilon(x)\|_{L^2_{k-1,\delta}} \leq C'\epsilon.
\]
holds for any $x \in \scrS_{k,\delta}\setminus \mathbb{B}(\eta)$.
 \end{enumerate} 
\end{lem}

\begin{proof}
Since $\scrS_{k,\delta}\setminus \mathbb{B}(\eta)$ has a free $Pin(2)$-action, we have the smooth Hilbert bundle
\[
\mathcal{E}:= (\scrS_{k,\delta}\setminus \mathbb{B}(\eta)) \times_{Pin(2)} \scrH_{k-1,\delta}\to  (\scrS_{k,\delta}\setminus \mathbb{B}(\eta))/Pin(2).
\]
Sicne $\mu: \scrS_{k,\delta}\setminus \mathbb{B}(\eta) \to \scrH_{k-1,\delta}$ is a $Pin(2)$-equivariant map, 
$\mu$ determines a section $\mu' : (\scrS_{k,\delta}\setminus \mathbb{B}(\eta))/Pin(2)\to \mathcal{E}$.
The section $\mu'$ is a smooth Fredholm section and the set $\mu'^{-1}(0)$ is compact by \cref{cpt}.
Now we consider a construction used in Ruan's virtual neighborhood technique~\cite{Ru98}.
Let $\gamma \in \mu'^{-1}(0)$.
The differential 
\[
d\mu'_{\gamma} : T_\gamma (\scrS_{k,\delta}\setminus \mathbb{B}(\eta))/Pin(2)  \rightarrow \scrH_{k-1,\delta}
\]
is a linear Fredholm map, and therefore there exist a natural number $n_{\gamma}$ and a linear map $f_{\gamma} : \R^{n_{\gamma}} \to  \scrH_{k-1,\delta}$ such that
\[
d\mu'_{\gamma} + f_{\gamma}  : T_\gamma (\scrS_{k,\delta}\setminus \mathbb{B}(\eta))/Pin(2) \oplus \R^{n_{\gamma}}  \to \scrH_{k-1,\delta}
\]
is surjective.
Concretely, we can give the map $f_{\gamma}$ as follows.
Let $V_{\gamma}$ be a direct sum complement in $\scrH_{k-1,\delta}$ of the image of $d\mu'_{\gamma}$.
By taking a basis of $V_{\gamma}$, we get a linear embedding $f_{\gamma} : \R^{n_{\gamma}} \to V_{\gamma} \subset \scrH_{k-1,\delta}$, where $n_{\gamma}=\dim V_{\gamma}$.
We here show that, by replacing $f_{\gamma}$ appropriately, we can assume that any element of $\im{f_{\gamma}}$ is smooth and has compact support.
For each member of the fixed basis of $V_{\gamma}$, we can take a sequence of smooth and compactly supported sections of $\Lambda^+(Z) \oplus S^{-}$ which converses to the member in the $L^{2}_{k-1,\delta}$ sense.
Then we get a sequence of maps $\{f_{\gamma,l}\}_{l}$ approaching $f_{\gamma}$ through the same procedure of the construction of $f_{\gamma}$ above.
Since surjectivity is an open condition, for a sufficiently large $l$, by replacing $f_{\gamma,l}$ with $f_{\gamma}$ we can assume that any element of $\im{f_{\gamma}}$ is smooth and has compact support.

For each $\gamma \in \mu'^{-1}(0)$, since surjectivity is an open condition, there exists a small open neighborhood $U_{\gamma}$ of $\gamma$ in $(\scrS_{k,\delta}\setminus \mathbb{B}(\eta))/Pin(2)$ such that $d\mu'_{\gamma'} + f_{\gamma} $ is surjective for any $\gamma' \in U_{\gamma}$.
Since $\mu'^{-1}(0)$ is compact,
there exist finitely many points $\gamma_{1}, \ldots, \gamma_{p} \subset \mu'^{-1}(0)$ such that $\mu'^{-1}(0) \subset \bigcup_{i=1}^{p}U_{\gamma_{i}}$.
Here we regard $\bigcup_{i=1}^{p}U_{\gamma_{i}}$ as an open submanifold of the infinite-dimensional manifold $(\scrS_{k,\delta}\setminus \mathbb{B}(\eta))/Pin(2)$.
(For the existence of this partition of unity on the infinite-dimensional manifold, see, for example, Chapter II, Corollary~3.8 of \cite{L95}.)
Set 
\[
n_{i} := n_{\gamma_{i}}, \ f_{i} := f_{\gamma_{i}},\  U_{i} := U_{\gamma_{i}}, \text{ and } n := \sum_{i=1}^{p}n_{i}.
\]
We fix a smooth partition of unity $\{\rho_{i} : U_{i} \to [0,1]\}_{i}$ subordinate to the open covering $\{U_{i}\}_{i=1}^{p}$ of $\bigcup_{i=1}^{p}U_{i}$.
Note that, until this point, we have not used $\epsilon$.
We here define a section
\[
\bar{g}_\epsilon : (\scrS_{k,\delta}\setminus \mathbb{B}(\eta))/Pin(2)  \times \R^{n} \to \mathcal{E}
\]
by
\begin{align} \label{epsilon}
\bar{g}_\epsilon(x,v)  :=\epsilon  \sum_{i=1}^{p} \rho_{i}(x) f_{i}(v_{i}),
\end{align}
where
\begin{align*}
&(x,v)=(x,(v_{1}, \ldots, v_{p}))\\
\in &(\scrS_{k,\delta}\setminus \mathbb{B}(\eta))/Pin(2)  \times \R^{n} = (\scrS_{k,\delta}\setminus \mathbb{B}(\eta))/Pin(2)  \times \R^{n_{1}} \times \cdots \times \R^{n_{p}}.
\end{align*}
One can easily check that, for any $\gamma \in \mu'^{-1}(0)$, the differential $d(\mu'+ \bar{g}_\epsilon)_{(\gamma,0)}$ is surjective.
Since any element of the image of $f_{i}$'s are smooth and has compact support, any element of the image of $\bar{g}_\epsilon$ is and does also.
Note that $\bar{g}_\epsilon(\gamma,0)=0$ holds for any $\gamma \in \mu'^{-1}(0)$.
Since surjectivity is an open condition, there exists an open neighborhood $\mathcal{N}$ of $\mu'^{-1}(0)$ in $(\scrS_{k,\delta}\setminus \mathbb{B}(\eta))/Pin(2)  \times \R^{n}$ such that, for any point $z \in \mathcal{N}$, the linear map $d(\mu'+\bar{g}_\epsilon)_z$ is surjective.
Because of the implicit function theorem, we can see that the subset
\[
\mathcal{U}:= \Set{ (x,v) \in   \mathcal{N} |( \mu + \bar{g}_\epsilon) (x,v) = 0   }
\]
of $\calN$, called a virtual neighborhood, has a structure of a finite dimensional manifold.
By Sard's theorem, the set of regular values of the map ${\rm pr}: \mathcal{U} \to \R^{n}$ defined as the restriction of the projection map
\[
{\rm pr}: (\scrS_{k,\delta}\setminus \mathbb{B}(\eta))/Pin(2)  \times \R^{n} \to \R^{n}
\]
is a dense subset of $\R^{n}$. Now we choose a regular value $v\in \R^n$ with the sufficiently small norm such that
\[
\set {x \in (\scrS_{k,\delta}\setminus \mathbb{B}(\eta))/Pin(2)| (\mu+ \bar{g}_\epsilon)(x,v)=0} \times \{v\} \subset  \mathcal{N}.
\]
Define
\[
g'_\epsilon : (\scrS_{k,\delta}\setminus \mathbb{B}(\eta))/Pin(2)\to \mathcal{E}
\]
by $g'_\epsilon(x) := \bar{g}_\epsilon(x,v)$.
Then we get a $Pin(2)$-equivariant map
\[
g_\epsilon: \scrS_{k,\delta}\setminus \mathbb{B}(\eta) \to \scrH_{k-1,\delta} 
\]
 by considering the pull-back of $g'_\epsilon$ by the quotient maps 
 \begin{align*}
  \begin{CD}
  (\scrS_{k,\delta}\setminus \mathbb{B}(\eta)) \times\scrH_{k-1,\delta}  @>{}>>  \mathcal{E}=(\scrS_{k,\delta}\setminus \mathbb{B}(\eta)) \times_{Pin(2)}  \scrH_{k-1,\delta}   \\
  @V{}VV    @VVV \\
\scrS_{k,\delta}\setminus \mathbb{B}(\eta) @>>>   (\scrS_{k,\delta}\setminus \mathbb{B}(\eta))/Pin(2)
  \end{CD}
\end{align*}
and composing the projection $(\scrS_{k,\delta}\setminus \mathbb{B}(\eta)) \times \scrH_{k-1,\delta} \to \scrH_{k-1,\delta}$.
The surjectivity of $d(\mu'+\bar{g}_\epsilon)$ ensures that this $g_{\epsilon}$ satisfies the first required condition in the statement of the \lcnamecref{per}. Since any element of the image of $\bar{g}_\epsilon$ is smooth and has compact support, the map $g_\epsilon$ satisfies the same condition.
This implies that $g_{\epsilon}$ meets the second and third conditions in the statement.
The fourth and fifth conditions follow from the expression \eqref{epsilon}.
\end{proof}

\subsection{Kuranishi model}
To obtain the $10/8$-inequality, we study a neighborhood of the reducible configuration and use the $Pin(2)$-equivariant  Kuranishi model for this.
We first recall the following well-known theorem.
(For example, see Theorem~A.4.3 in \cite{MS12}.)

\begin{theo}[Kuranishi model]
\label{Kuranishi}
Let $G$ be a compact Lie group and $V$ and $V'$ be Hilbert spaces equipped with smooth $G$-actions and $G$-invariant inner products.
Suppose that there exists a $G$-equivariant smooth map $f : V \to V'$ with $f(0)=0$ and such that 
$df_0:  V \to V'$ is Fredholm.
Then the following statements hold.
\begin{enumerate}
\item There exists a $G$-invariant open subset $U \subset V$ and a $G$-equivariant diffeomorphism $T: U \to T(U) $ satisfying the following conditions:
\begin{itemize} 
\item $T(0)=0$.
\item There exist a $G$-equivariant linear isomorphism $D: (\Ker df_0)^{\perp}\to  \im df_0$ and a smooth $G$-equivariant map $\tilde{f}: V \to (\im df_0)^{\perp} $ such that the map
\[
f\circ T: U \to T(U) \subset V \to V'
\]
is written as $f\circ T(v,w) = (Dw, \tilde{f}(v,w)) \in \im df_0\oplus (\im df_0)^{\perp} = V'$ for $(v,w) \in \Ker df_0\oplus (\Ker df_0)^{\perp}=V$.
\item If we define $F: \Ker df_0 \cap U\to (\im df_0)^{\perp}$ by $F(v):= \tilde{f}(v,0)$, then $F^{-1}{(0)}$ can be identified with $U\cap f^{-1}(0)$ as $G$-spaces.
\end{itemize}
\item For a real number $c$ satisfying $0 < c \leq 1/2$, let $U_c(f)$ be the open set in $V$ defined by
\begin{align}
\label{Kuranishi open}
U_c(f) = \set{ x \in V | \| {\rm pr}_{(\Ker df_0)^{\perp}}- D^{-1}\circ {\rm pr}_{\im df_0} \circ df_x \|_{\mathcal{B}(V, V)}<c },
\end{align}
where  ${\rm pr}_{W}$ is the projection to a subspace $W$ and $\| - \|_{\mathcal{B}(V, V)}$ is the operator norm.
Let $\Psi : V \to V$ be the map defined by
\[
\Psi(x) := x + D^{-1}\circ {\rm pr}_{\im df_0}(f(x) - df_{0}(x)).
\]
Then, the image $\Psi(U_{c}(f))$ is an open subset of $V$ and the restriction $\Psi|_{U_{c}(f)} :U_{c}(f) \to \Psi(U_{c}(f))$ is a diffeomorphism.
\item
As the open set $U$ in (1), we can take any open ball centered at the origin and contained in $\Psi(U_{c}(f))$.
\end{enumerate}
\end{theo}

\subsection{Spin $\Gamma$-structure on the \SW moduli space}
\label{subsection Spin Gamma structure of the SW moduli space}

In this subsection, we show that there is a natural spin $\Gamma$-structure (equivariant spin structure) on $\widetilde{\M }_{k,\delta}(Z)$. 
To show this, we need several definitions related to a $Pin(2)$-equivariant version of the index of a family of Fredholm operators.
A non-equivariant version of the argument of this \lcnamecref{subsection Spin Gamma structure of the SW moduli space} was originally considered by H.~Sasahira~\cite{MR2284407}.

Let $G$ be a compact Lie group.
\begin{defi} Let $H_1$ and $H_2$ be separable Hilbert spaces with linear $G$-actions. 
Let $\text{Fred}(H_1,H_2)$ be the set of Fredholm operators from $H_1$ to $H_2$. We define a topology on $\text{Fred}(H_1,H_2)$ by the operator norm and an action of $G$ on $\text{Fred}(H_1,H_2)$ by $f \mapsto g^{-1} f g$ , where $f \in \text{Fred}(H_1,H_2)$ and $g \in G$.
\end{defi}
As in the non-equivariant case, for a compact $G$-space $K$, there is a map: 
\[
\ind_K: [K, \text{Fred}(H_1,H_2)]_G \to KO_G(K)
\]
via the equivariant families index, where $[K, \text{Fred}(H_1,H_2)]_G$ is the set of $G$-homotopy classes of $G$-maps from $K$ to $\text{Fred}(H_1,H_2)$.

In this subsection, for fixed $\eta>0$ and $\epsilon>0$, we choose a perturbation $g_\epsilon$ as in \cref{per}.
We also fix a $Pin(2)$-equivariant cut-off function
 $\rho: \scrS_{k,\delta}\to [0,1]$ satisfying
 \begin{align}
 \label{def of rho}
 \rho(x) = 
 \begin{cases} 
 0, \ \text{if } x \in \mathbb{B}(\eta),\\ 
 1 , \ \text{if } x \in \mathbb{B}(2\eta)^c,
 \end{cases}
 \end{align}
 where $\mathbb{B}(2\eta)^c$ is the complement of $\mathbb{B}(2\eta)$ in $\scrS_{k,\delta}$.
(To construct such a function, we use a map induced by the square of the $L^2_{k,\delta}$-norm on $\scrS_{k,\delta}$.)
 We have the $Pin(2)$-equivariant smooth map 
\[
 g_\epsilon :  \scrS_{k,\delta}\setminus \mathbb{B}(\eta)  \to \scrH_{k-1,\delta}
 \]
given in \cref{per}.
We now consider the following map 
\[
\mu_\epsilon:= \mu+ \rho  g_\epsilon: \scrS_{k,\delta}  \to \scrH_{k-1,\delta}.
\]
In our situation, we put $H_1=\scrS_{k,\delta}$, $H_2=\scrH_{k-1,\delta}$, $G=Pin(2)$ and let $K$ be an $G$-invariant topological subspace of $H_1$.
The Fredholm maps $d(\mu_\epsilon)_x : H_1 \to H_2$ for $x \in K$ determine a class
\[
[ \{d(\mu_\epsilon)_x\}_{x\in K}] \in [K, \text{Fred}(H_1,H_2)]_{Pin(2)}.
\]
If $K=H_1$, then we have an isomorphism
\[
[H_{1}, \text{Fred}(H_1,H_2)]_{Pin(2)} \cong [\{0\}, \text{Fred}(H_1,H_2)]_{Pin(2)}
\]
via a $Pin(2)$-equivariant deformation retraction from $H_{1}$ to $\{0\}$.
This isomorphism implies that
\begin{align}
\label{eq: for cal of fam ind}
[ \{d(\mu_\epsilon)_x\}_{x \in H_{1}}]
=[ \{d\mu_0\}_{x \in H_{1}}]
\end{align}
in $[H_{1}, \text{Fred}(H_1,H_2)]_{Pin(2)}$ since $d(\mu_\epsilon)_0 = d\mu_{0}$.

It follows from \cref{perturbation cpt}, shown in the next subsection, that $\mu_\epsilon^{-1}(0)\setminus \mathbb{B}(\eta)$ is compact for an appropriate choice of $\delta$.
In the rest of this subsection, we shall use this comactness.
For a given $Pin(2)$-space $K$ and a $Pin(2)$-module $W$,  we denote by $\underline{W}$ the product $Pin(2)$-bundle on $K$ with fiber $W$.

\begin{lem} \label{KOclass}
The class $[T(\mu_\epsilon^{-1}(0)\setminus \mathbb{B}(2\eta))] \in KO_{Pin(2)}(\mu^{-1}(0)\setminus \mathbb{B}(2\eta))$ is equal to $[\underline{\Ker d \mu_0}]- [\underline{\Coker d\mu_0}]$.
\end{lem}
\begin{proof}
Set $K = \mu_{\epsilon}^{-1}(0)\setminus \mathbb{B}(\eta)$.
Then we have
\begin{align}
\label{eq: TK ind}
[TK] = \ind_{}[\{d(\mu_\epsilon)_x\}_{x\in K}] \in KO_{Pin(2)}(K).
\end{align}
Note that the inclusion $i: K \to  H_1$ induces the map 
\[
i^* : [H_1, \text{Fred}(H_1,H_2)]_{Pin(2)} \to [K,\text{Fred}(H_1,H_2)]_{Pin(2)}.
\]
Using the equality \eqref{eq: for cal of fam ind}, one can check 
\begin{align}
\label{eq: dmuzero and dmuepsilon}
[ \{d\mu_0\}_{x \in K}] = i^*[ \{d\mu_0\}_{x \in H_{1}}]=i^* [\{d (\mu_\epsilon)_{x}\}_{x \in H_{1}}]=[\{d (\mu_\epsilon)_{x}\}_{x \in K}].
\end{align}
The index of the left-hand-side is given as
\begin{align}
\label{eq: ind of dmuzero}
\ind[ \{d\mu_0\}_{x \in K}]=[\underline{\Ker d \mu_0}]- [\underline{\Coker d\mu_0}].
\end{align}
The equalities \eqref{eq: TK ind}, \eqref{eq: dmuzero and dmuepsilon}, and \eqref{eq: ind of dmuzero} prove the lemma.
\end{proof}

\begin{cor}
\label{stably}
Under the same assumption of \cref{KOclass}, there exists a $Pin(2)$-module $W$ such that 
\[
T(\mu_\epsilon^{-1}(0)\setminus \mathbb{B}(\eta)) \oplus  \underline{\Coker d \mu_0} \oplus  \underline{W} \cong  \underline{\Ker d\mu_0}  \oplus  \underline{W}
\]
as $Pin(2)$-bundles. 
\end{cor}
By the use of  \cref{stably}, we can equip the \SW moduli space with a structure of spin $\Gamma$-manifold.

\begin{cor} \label{spin G mfd str}
Under the same assumption of \cref{KOclass} and the condition $b^+(M)$ is even, 
$\mu_\epsilon^{-1}(0)\setminus \mathbb{B}(\eta)$ has a structure of a spin $\Gamma$-manifold. 
\end{cor}
\begin{proof}
 By applying \cref{stably}, we have a $Pin(2)$-module $W$ satisfying 
\begin{align} \label{isom}
T(\mu_\epsilon^{-1}(0)\setminus \mathbb{B}(\eta)) \oplus  \underline{\Coker d \mu_0} \oplus  \underline{W} \cong   \underline{\Ker d \mu_0} \oplus \underline{W}.
\end{align}
By \cref{Pin(2)-module structure}, we can assume that $W$ has a real $Pin(2)$-module structure. 
We regard $Pin(2)$-modules as $\Gamma$-modules and $Pin(2)$-spaces as $\Gamma$-spaces via \eqref{definition of Gamma}. 
Since $b^+(M)$ is even, the dimensions of $\Coker d \mu_0$ and $\Ker d \mu_0$ are even. By \cref{spin G module tilR} and \cref{spin G module H}, $ \Coker d \mu_0$ and  $\Ker d \mu_0$ have spin $\Gamma$-module structures. 
The spin $\Gamma$-module structures on $ \Coker d \mu_0$, $\Ker d \mu_0$ and $W$ determine spin structures on  $\underline{\Coker d \mu_0}/\Gamma$, $\underline{\Ker d \mu_0}/\Gamma$ and $\underline{W}/\Gamma$ as vector bundles on $(\mu_\epsilon^{-1}(0)\setminus \mathbb{B}(\eta) )/\Gamma $ by \cref{real spin module}.

The isomorphism \eqref{isom} gives the isomorphism: 
\begin{align*}
T(\mu_\epsilon^{-1}(0)\setminus \mathbb{B}(\eta))/\Gamma \oplus  \underline{\Coker d \mu_0}/\Gamma
\oplus  \underline{W}/\Gamma 
 \cong   \underline{\Ker d \mu_0}/\Gamma \oplus \underline{W}/\Gamma.
\end{align*}
Since $ \underline{\Ker d \mu_0}/\Gamma \oplus \underline{W}/\Gamma$ has a spin structure induced by the spin structures on  $\underline{\Ker d \mu_0}/\Gamma$ and $\underline{W}/\Gamma$, $T(\mu_\epsilon^{-1}(0)\setminus \mathbb{B}(\eta))/\Gamma$ also admit a spin structure. By \cref{equiv spin2}, we obtain a structure of a spin $\Gamma$-manifold on $\mu_\epsilon^{-1}(0)\setminus \mathbb{B}(\eta)$.
\end{proof}
\begin{rem}\label{indep on W}
One can check that the spin $Pin(2)$-structure on $\mu_\epsilon^{-1}(0)\setminus \mathbb{B}(\eta)$ in  \cref{spin G mfd str} does not depend on the choice of $W$.
\end{rem}

\subsection{Main construction}

The following theorem contains the main construction of this paper.

\begin{theo}
\label{main construction} 
Under the assumption of \cref{main} and the condition that $b^+(M)$ is even, there exist real spin $\Gamma$-modules $U_0$ and $U_1$ with $\Gamma$-invariant norms and $\Gamma$-equivariant smooth map $\phi : S(U_0) \to U_1$ from the unit sphere of $U_{0}$ satisfying the following conditions:
\begin{itemize}
\item The group $\Gamma$ acts freely on $U_0 \setminus \{0\}$.
\item The map $\phi : S(U_0) \to U_1$ is transverse to $0 \in U_{1}$.
\item The $\Gamma$-manifold $\phi^{-1}(0)$ bounds a compact manifold with a free $\Gamma$-action, making it a spin $\Gamma$-manifold. 
\item As a $\Gamma$-representation space, $U_i$ is isomorphic to $\tilde{\R}^{l_i} \oplus \quat^{m_i}$ for $i=0$ and $1$, where $l_0=0,  \ l_1=b^+(M)$ and $2m_0-2m_1= - \ind_{\C} D_{A_0}$,
\end{itemize}
where the definition of $\Gamma$ is given in \eqref{definition of Gamma}.
\end{theo}

\begin{proof}
Let $X$, $Y$, $M$ and $Z$ be as in \cref{section Preliminaries}.
Suppose that $X$ admits a PSC metric $g_X$.
We fix a positive number $\delta$ satisfying $\delta < \min \{ \delta_0\ , \delta_1\ , \delta_2 \}$.
(Recall that $\delta_{0}$, $\delta_{1}$, and $\delta_{2}$ are given in \cref{rem deltazero}, \cref{formal dimension}, and \cref{cpt} respectively.)
We denote by $D_{0}$ the operator $d\mu_0$ and put $W_0=\Ker D_0$ and $W_1=\im D_0$. 
Since the operator $D_0 :W_0^{\perp} \to W_1$ is an isomorphism, there exists the inverse map $D_0^{-1}:W_1\to W_0^{\perp}$, where $\perp$ is the orthogonal complement with respect to $L^2_{k,\delta}$-norm in $\scrS_{k,\delta} $.
We use \cref{Kuranishi} for the following setting: 
\[
V=\scrS_{k,\delta} , \   V'= \scrH_{k-1,\delta} , \ G=Pin(2),  \text{ and } \phi = \mu.
\]
Then we get a Kuranishi model for $\mu$ near the reducible.
For this model, we use the open subset $U_c(\mu) \subset V$ for $c$ with $0 <c <  \min \{ \frac{1}{8}, \frac{1}{8} \|D_0^{-1} \|_{\mathcal{B}(W_{1},W_0^{\perp})}^{-1} \}$ defined as in \eqref{Kuranishi open}, where  $\|D_0^{-1} \|_{\mathcal{B}(W_1,W_0^{\perp})}$ is the operator norm of $D_0^{-1}$.
We fix a positive real number $\eta$ satisfying
\begin{align}
\label{choice of eta}
\mathbb{B}(4\eta ) \subset U_{c}(\mu) \subset \scrS_{k,\delta}
\end{align}
and also fix a $Pin(2)$-equivariant cut-off function
 $\rho: \scrS_{k,\delta}\to [0,1]$ as \eqref{def of rho}.
For any $\epsilon>0$, we have the $G$-equivariant smooth map 
\[
 g_\epsilon :  \scrS_{k,\delta}\setminus \mathbb{B}(\eta)  \to \scrH_{k-1,\delta}
 \]
given in \cref{per}, and can consider the map 
\[
\mu_\epsilon = \mu+ \rho  g_\epsilon: \scrS_{k,\delta}  \to \scrH_{k-1,\delta}
\]
as in \cref{subsection Spin Gamma structure of the SW moduli space}.
Note that the map $\mu_\epsilon$ is a smooth $Pin(2)$-equivariant Fredholm map.
Because of \cref{per}, the differential
$d(\mu_\epsilon)_x$ is surjective for any $x \in \mu_{\epsilon}^{-1}(0) \cap \mathbb{B}(2\eta)^c$, and therefore $\mu_{\epsilon}^{-1}(0) \cap \mathbb{B}(2\eta)^c$ is a finite dimensional manifold.
We also note that $\mu_{\epsilon} = \mu$ on $\mathbb{B}(\eta)$.
We define $\Psi_{\mu_{\epsilon}} : V \to V$ by
\begin{align*}
\Psi_{\mu_{\epsilon}}(x)
&= x + D^{-1}_{\epsilon} \circ {\rm pr}_{\im (d(\mu_{\epsilon})_{0})}(\mu_{\epsilon}(x) - d(\mu_{\epsilon})_{0}(x))\\
&= x + D^{-1}_{0} \circ {\rm pr}_{\im D_{0}}(\mu_{\epsilon}(x) - D_{0}(x)),
\end{align*}
where $D_{\epsilon} = d(\mu_{\epsilon})_{0}$, which is just $D_{0}$.
We now use the following \lcnamecref{perturbation cpt}:

\begin{lem}
\label{perturbation cpt}
For $\delta$  with $0<\delta < \delta_2$, the space $\mu_\epsilon^{-1}(0)/S^1$ is compact, where $\delta_2$ is the constant given in \cref{cpt}.
\end{lem}

The proof of this \lcnamecref{perturbation cpt} is given at the end of this subsection.
Assuming \cref{perturbation cpt}, then the space $\mu_\epsilon^{-1}(0)\setminus \mathbb{B}(\eta)$ is also compact.
Next, we consider a neighborhood of the reducible.
Applying \cref{Kuranishi} for $f= \mu_\epsilon$, we obtain a Kuranishi model for $\mu_\epsilon$ near the reducible.
Here we use the open subset $U_c(\mu_\epsilon)$ defined as in \eqref{Kuranishi open} for this Kuranishi model.
Fix a positive number $c'$ satisfying  $  c  <\frac{1}{2}c' < \frac{1}{8}$.
We here choose $\epsilon$ so that  
\[
 \| D_0^{-1}\|_{\mathcal{B}(W_1,W_0^{\perp})} ( C\max |d \rho | \epsilon + C' \epsilon )   < \frac{1}{2}c',
\]
where $C$ and $C'$ are the constants in \cref{per}.
Then 
\begin{align}
\label{inclusion rel Uc Ucp}
U_{c} (\mu ) \subset U_{c'} (\mu_\epsilon ) 
\end{align}
holds, because for $x \in U_{c} (\mu )$ we have 
\begin{align}
\begin{split}
\label{ineq ucmu}
&\|\text{pr}_{(\ker D_0)^{\perp}} -  D_0^{-1}\circ \text{pr}_{\im D_0}\circ d(\mu_\epsilon)_x \|_{ \mathcal{B}(V,V')}\\
\leq  &\|\text{pr}_{(\ker D_0)^{\perp}} - D_0^{-1}\circ \text{pr}_{\im D_0}\circ( d\mu_x+ d(\rho g_\epsilon)_x) \|_{ \mathcal{B}(V,V')}\\
\leq  &c+\|D_0^{-1}\circ \text{pr}_{\im D_0}\circ d(\rho g_\epsilon)_x) \|_{ \mathcal{B}(V,V')}\\
\leq  &c+\|D_0^{-1}\|_{\mathcal{B}(W_{1},W_{0}^{\perp})} ( C\max |d \rho | \epsilon + C' \epsilon )<c'.
\end{split}
\end{align}
Here we use the definition of $U_{c}(\mu)$ given in \cref{Kuranishi} in the second inequality and \cref{per} in the last inequality.

Next, we show
\begin{align}
\label{inclusion rel B3 Phi}
\mathbb{B}(3\eta) \subset \Psi_{\mu_{\epsilon}}(U_{c'} (\mu_\epsilon)).
\end{align}
We first note that the argument to get the inequality \eqref{ineq ucmu} also shows that $\|\id - d(\Phi_{\mu_{\epsilon}})_{x}\| < c'$ for any $x \in \mathbb{B}(4\eta)$.
Going back to a proof of the inverse function theorem, this inequality implies that $\mathbb{B}(4\eta(1-c')) \subset \Psi_{\mu_{\epsilon}}(\mathbb{B}(4\eta))$.
(For example, see Lemma~A.3.2 in \cite{MS12}.)
Using \eqref{choice of eta}, \eqref{inclusion rel Uc Ucp}, we get $\mathbb{B}(4\eta(1-c'))\subset  \Psi_{\mu_{\epsilon}}(U_{c'}(\mu_{\epsilon}))$.
Thus we have \eqref{inclusion rel B3 Phi}.

Now let us recall \cref{Kuranishi}.
Because of \eqref{inclusion rel B3 Phi}, \cref{Kuranishi} ensures that there exists a $Pin(2)$-equivariant diffeomorphism $T: \mathbb{B}(3\eta) \to T(\mathbb{B}(3\eta)) $ satisfying the following conditions: 
\begin{itemize} 
\item $T(0)=0$.
\item The map
\[
\mu_\epsilon \circ T: \mathbb{B}(3\eta) \to T(\mathbb{B}(3\eta)) \subset \scrS_{k,\delta}  \to \scrH_{k-1,\delta}
\]
is given by $(v,w) \mapsto (Dw, \tilde{\mu}_{\epsilon}(v,w))$
via the decompositions $\scrS_{k,\delta} = \Ker d(\mu_\epsilon)_0\oplus (\Ker d(\mu_\epsilon)_0)^{\perp}$ and $\scrH_{k-1,\delta} = \im d(\mu_\epsilon)_0\oplus (\im d(\mu_\epsilon)_0)^{\perp}$.
Here
\[
D: (\Ker d(\mu_\epsilon)_0)^{\perp}\to \im d(\mu_\epsilon)_0
\]
is a $Pin(2)$-equivariant linear isomorphism and
\[
\tilde{\mu}_\epsilon: \scrS_{k,\delta} \to (\im d(\mu_\epsilon)_0)^{\perp}
\]
is a smooth $Pin(2)$-equivariant map.
\item Define $\mu^*_\epsilon : \Ker d(\mu_\epsilon)_0 \cap \mathbb{B}(3\eta) \to (\im d(\mu_\epsilon)_0)^{\perp}$
by $\mu^*_\epsilon(v):=\tilde{\mu}_\epsilon(v,0)$.
Then $\mathbb{B}(3\eta) \cap (\mu^*_\epsilon)^{-1}{(0)}$ can be identified with $\mathbb{B}(3\eta) \cap (\mu_\epsilon)^{-1}(0)$ as $Pin(2)$-spaces. 
\end{itemize}
We denote the identification between $\mathbb{B}(3\eta) \cap (\mu^*_\epsilon)^{-1}{(0)}$ and $\mathbb{B}(3\eta) \cap (\mu_\epsilon)^{-1}(0)$ by $\Phi: \mathbb{B}(3\eta) \cap (\mu^*_\epsilon)^{-1}{(0)} \to \mathbb{B}(3\eta) \cap (\mu_\epsilon)^{-1}(0)$.
Since $\mu_\epsilon(x)= \mu(x)$ for $x \in \mathbb{B}(\eta)$,  
we can see that 
\begin{align}\label{dimension}
\Ker d(\mu_\epsilon)_0 \cong \Ker D_{A_0}, \ \Coker d(\mu_\epsilon)_0 \cong \Coker D_{A_0}\oplus \tilR^{b^+}
\end{align}
as $Pin(2)$-modules by \cref{formal dimension}.
We set 
\[
U_0 := \Ker d(\mu_\epsilon)_0 \text{ and } \ U_1 := \Coker d(\mu_\epsilon)_0.
\]
 By the construction of the Kuranishi model, the equalities $\Ker (d\mu_\epsilon)_x \cong \Ker (d\mu^*_\epsilon)_{\Phi (x)} $ and $\Coker (d\mu_\epsilon)_x \cong \Coker (d\mu^*_\epsilon)_{\Phi (x)} $ hold for any $x\in \mathbb{B}(3\eta) \cap (\mu_\epsilon)^{-1}(0)$. 
Therefore, there exists a positive real number $\upsilon$ such that $(\mu^*_\epsilon)^{-1}(0) \setminus \mathbb{B}(3\eta- \upsilon)$ has the structure of a $Pin(2)$ finite dimensional manifold.
Here let us consider the $Pin(2)$-invariant smooth map 
\[
\psi :(\mu^*_\epsilon)^{-1}(0) \setminus \mathbb{B}(3\eta- \upsilon)   \to (0, \infty)
\]
defined by $\psi (x) := \| x\|_{L^2_{k,\delta}}^2$.
Sard's theorem implies that there exists a dense subset $S$ in $(0, \infty)$ such that $s$ is a regular value of $\psi$ for any $s \in S$. Now we fix a regular value $s \in S$. 
Then the space $\psi^{-1}([s, \infty))   \cup_{ \Phi \circ T} (  \mu_\epsilon^{-1} (0) \setminus \mathbb{B}(3\eta- \upsilon)) $ has a structure of a compact $Pin(2)$-manifold with boundary. 
We equip $U_0$ with the norm defined by
\[
\|v\| := \frac{1}{\sqrt{s}} \|v\|_{L^2_{k,\delta}}
\]
and $U_1$ with that defined as the restriction of the $L^2_{k-1,\delta}$-norm. 

We set  $\phi := \mu^*_\epsilon|_{S (U_0)}$.
We regard $Pin(2)$-modules as $\Gamma$-modules and $Pin(2)$-spaces as $\Gamma$-spaces via \eqref{definition of Gamma}.
Now we check that the conclusions of \cref{main construction} are satisfied. 
Since $b^+(M)$ is even, $\dim U_0$ and $\dim U_1$ are also even. By  \cref{spin G module tilR} and \cref{spin G module H}, $U_0$ and $ U_1$ admit real spin $\Gamma $-module structures.
The map 
\begin{align}\label{phi def}
\phi : S (U_0) \to U_1
\end{align}
is transverse to $0$ because of the choice of $s$.  
This implies the second condition. 
Since the $Pin(2)$-action on $\Ker d(\mu_\epsilon)_0 =\Ker D_{A_0}$ by quaternionic multiplication, the first condition follows. 
By the use of \cref{spin G mfd str}, we can equip a structure of a spin $\Gamma$-manifold on $\psi^{-1}([s, \infty))   \cup_{ \Phi \circ T} (  \mu_\epsilon^{-1} (0) \setminus \mathbb{B}(3\eta- \upsilon)) $. On the other hand, the differential of \eqref{phi def} gives an isomorphism
\[
\Ker d\phi \oplus \underline{U_1} \oplus \underline{\R} \cong \underline{U_0}.
\]
We equip $\R$ with the trivial real spin $\Gamma$-module structure. The the vector bundles $\underline{U_1}/ \Gamma$,  $\underline{\R}/ \Gamma$ and $\underline{U_0}/ \Gamma$ on $\phi^{-1} (0)/ \Gamma $ has spin structures by \cref{real spin module}. Therefore, $\phi^{-1}(0)/ \Gamma$ also admit a spin structure. This induces a real spin $\Gamma$-manifold structure on $\phi^{-1}(0)$. Since the constructions are same, the structure of a spin $\Gamma$-manifold on $\phi^{-1}(0)$ coincide with that of  $\partial (\psi^{-1}([s, \infty))   \cup_{ \Phi \circ T} (  \mu_\epsilon^{-1} (0) \setminus \mathbb{B}(3\eta- \upsilon)))$. This implies the third condition. The isomorphism \eqref{dimension} implies the fourth condition.
\end{proof}

At the end of this subsection, we give the proof of \cref{perturbation cpt}.

\begin{proof}[Proof of \cref{perturbation cpt}]
The proof is similar to that of \cref{cpt}. Let $\{[(A_n,\Phi_n)]\}$ be a sequence in $\nu_\epsilon^{-1}(0)/S^1$.
 For all $n$, the pair $(A_n,\Phi_n)$ satisfies the equation
\[
(\nu+\rho g_\epsilon) (A_n,\Phi_n) = 0.
\]
Because of the property of $g_\epsilon$ in \cref{per}, we have the inequality 
\[
\|g_\epsilon (A_n,\Phi_n) \|_{L^2_{k-1,\delta}} <C 
\]
for $(A_n,\Phi_n) \in \scrS_{k,\delta}$. Therefore the topological energy (see the proof of \cref{cpt}) of $(A_n,\Phi_n)$ is bounded by some positive number (independent of $n$) as in the proof of \cref{cpt}. 
Moreover, there exists a positive integer $N \gg 0$ satisfying $g_\epsilon (A) |_{W[N,\infty]}=0$ for $A \in \scrS_{k,\delta}$.
Therefore $(A_n, \Phi_n)$ satisfies the usual \SW equation on $W[N+1,\infty]$ for all $n$. There exist a subsequence $\{(A_{n'},\Phi_{n'})\}$ of $\{(A_n,\Phi_n)\}$ and gauge transformations $g_n$ on $W[N+2,\infty]$ such that $\{g_n^*(A_{n'},\Phi_{n'})\}$ converges on $W[N+2,\infty]$ as in the argument in \cref{cpt}. We should show the existence of a subsequence $\{(A_{n''},\Phi_{n''})\}$ of $\{(A_{n'},\Phi_{n'})\}$ and gauge transformations $h_n$ on $M \cup_Y W[0,N+3]$ satisfying that $\{h_n^*(A_{n''},\Phi_{n''})\}$ converges on $ M\cup W[0,N+3]$. It can be proved by essentially the same way as in Theorem 5.1.1 of \cite{KM07}. The key point is the boundedness of the analytical energies of $(A_{n'},\Phi_{n'})$. Finally, we paste $g_n$ and $h_n$ by some bump-functions and get the conclusion.
\end{proof}

\subsection{Completion of the proof of \cref{main}}
\label{sebsection: Completion of the proof of Main Thm}
In this subsection, we complete the proof of \cref{main}.

\begin{proof}[Proof of \cref{main}]
Let $Y$, $X$, $M$ and $Z$ be as in \cref{section Preliminaries}.
The proof is given according to three cases regarding $b^{+}(M)$.

The first case is when $b^+(M)$ is positive and even.
By applying \cref{main construction},  there exist real spin $\Gamma$-modules $U_0$ and $U_1$ with $\Gamma$-invariant norms and $\Gamma$-equivariant smooth map $\phi : S(U_0) \to U_1$ from the unit sphere of $U_{0}$ satisfying the following conditions:
\begin{itemize}
\item The group $\Gamma$ acts freely on $U_0 \setminus \{0\}$.
\item The map $\phi : S(U_0) \to U_1$ is transvers to $0 \in U_{1}$.
\item The $\Gamma$-manifold $\phi^{-1}(0)$ bounds a compact manifold acted by $\Gamma$ freely, as a spin $\Gamma$-manifold. 
\item As $\Gamma$-representation spaces, $U_i$ is isomorphic to $\tilde{\R}^{l_i} \oplus \quat^{m_i}$ for $i=0$ and $1$, where $l_0=0,  \ l_1=b^+(M)$ and $2m_0-2m_1= - \ind_{\C} D_{A_0}$.
\end{itemize}
By the fourth condition we can write
\[
U_0 =  \quat^{m_0}, \ U_1 = \tilde{\R}^{l_1} \oplus \quat^{m_1}.
\]
On the other hand, we have a smooth map $\phi : S(U_0) \to U_1$ which is  transverse to $0$. 
By the definition, we have the equality
\[
w(U_0,U_1) = [\phi^{-1}(0)]  \in \Omega^{\text{spin}}_{\Gamma ,\text{free}},
\]
By the third condition, the class $w(U_0,U_1)=0$. 
Therefore we apply \cref{Kametani} and obtain an element $\alpha \in  KO_{\Gamma}^{r_1-r_0}(pt)$ such that 
\[
e( \tilde{\R}^{l_1})e( \quat^{m_1}) =e( \tilde{\R}^{l_1} \oplus \quat^{m_1}) = \alpha e(\quat^{m_0}),
\]
where $\dim U_0 =4m_0=  r_0$ and $\dim U_1= 4m_1+l_1 = r_1$.
Now we apply \cref{Furuta Kametani} and get the inequality
\[
0\leq 2(m_1-m_0) +l_1-2.
\]
This implies that 
\[
 \ind_{\C} D^+_{A_0} +2 \leq b^+(M).
 \]
 By combining \cref{cor of W formula} and \cref{Lin}, we have 
 \begin{align}
 \label{eq; ineq case 1}
 h(Y,\frakt) - \frac{\sigma (M) }{8}+2 \leq b^+(M).
 \end{align}
 
The second case is when $b^+(M)$ is odd.
In this case we can obtain the estimate \eqref{eq; ineq case 1} for $M \# S^2 \times S^2$ instead of $M$ itself.
This implies that
\[
h(Y,\frakt)-\frac{\sigma (M)}{8} +1 \leq b^+(M).
\]
 
Lastly we consider the case that  $b^+(M)=0$. In this case, by \cref{reducible}, we have an unique reducible element in $\M_{k,\delta,h}(Z)$ for any $h$.  We can choose $h$ such that $\widetilde{\M}_{k,\delta,h}(Z)$ has a structure of a $2\ind_\C D^+_{A_0}$-dimensional manifold. The proof is essentially the same as in the case for oriented closed $4$-manfolds. (See \cite{Mo03}) 
Since the $S^1$-action on the reducible $[(A_0, 0 ) ]$ is identified with the standard action $S^1$ on $\C^{\ind_\C D^+_{A_0}}$ by the Kuranishi model along $[(A_0, 0 ) ]$, we have a compact manifold $\widetilde{\M}_{k,\delta,h}(Z)/S^1$ with one singularity modeled on the cone on $\C P^{\ind_\C D^+_{A_0}-1}$ (see \cref{cpt}).
Then we have 
\[
2\ind_\C D^+_{A_0}   \leq 0 .
\]
 By combining \cref{cor of W formula} and \cref{Lin}, we have 
 \[
h(Y,\frakt)-\frac{\sigma (M)} {8}   \leq 0 = b^+(M).
\] 
Therefore we have the conclusion.
\end{proof}

\begin{rem}
D.~Veloso~\cite{Ve14} considered boundedness of the monopole map for periodic-end 4-manifolds which admit PSC metric on the ends.
It seems that this argument shall also provide a similar conclusion.
The authors would like to express their deep gratitude to Andrei Teleman for informing them of  Veloso's argument.
\end{rem}

\begin{rem} 
In \cite{FK05}, Furura--Kametani showed a $10/8$-type inequality which is stronger than usual 10/8 inequality (Theorem 1 in \cite{Fu01}).
Since their method uses the divisibility \eqref{divisibility}, by using our method, it seems that one can prove such a stronger type of the inequality.
\end{rem}

\section{Examples}\label{examples}
In this section, we first give several constructions of homology $S^1\times S^3$'s. We then give examples of oriented homology $3$-spheres satisfying $\psi([Y])>0$.  At the end of this section, we see a family of homology $S^1\times S^3$'s which cannot have PSC metrics and compare our method with several known obstructions to PSC metrics. 
Recall that, for an integral homology $3$-sphere, there uniquely exists a spin structure on it up to isomorphism, and if $Y$ is an integral homology $3$-sphere, we write $\psi(Y), \epsilon(Y)$ and $h(Y)$ for $\psi(Y,\frakt), \epsilon(Y,\frakt)$ and $h(Y,\frakt)$ for the unique spin structure $\frakt$ on $Y$.

\subsection{Examples of homology $S^1\times S^3$'s}
In this subsection, we construct families of homology $S^1\times S^3$'s. 
\subsubsection{Examples from algebraically canceling pairs}
Let $Y$ be a homology $3$-sphere and $G$ a finitely generated perfect group with a presentation 
\[
G \cong \langle x_1, \cdots, x_n\rangle  / \langle r_1, \cdots, r_m \rangle .
\]
Let us impose the following minimality condition on this presentation:
let $\pi : \langle x_1, \cdots, x_n\rangle \to \Z^{n}$ be the abelianization.
We suppose that $\pi(r_{1}), \ldots, \pi(r_{m})$ are linearly independent over $\Z$.
Take a framed $n$-component link  $L=k_1\cup \cdots \cup k_m$ in $Y$. Fix $1$-handles $\{h_i^1\}_{i=1}^n$ and 2-handles $\{h_i^2\}_{i=1}^m$ attached to $D^3$ in $Y$ representing generators $\{x_i\}_{i=1}^n$ and relations $\{r_j\}_{j=1}^m$. We take the connected sums of $k_i$ and the attaching spheres of $h^2_i$ and obtain new $2$-handles $\{\widetilde{h}_i^2\}_{i=1}^m$ attached to $ Y$. We call $H =(\{h_i^1\}_{i=1}^n, \{\widetilde{h }_i^2\}_{i=1}^m)$ an {\it algebraically canceling pair}. 
We denote by $W(Y, H)$ the trace of the surgery of $Y$ along $H$ and by $Y'$ the boundary component of $W(Y,H)$ which is different from $Y$. Note that $W(Y, H)$ gives a homology cobordism from $Y$ to $Y'$.
Taking the double of $W(Y,H)$, we obtain a homology $S^1\times S^3$ denoted by $X(Y,  H)$.
 
The $1$- and $2$-handle pair of figure~2 in \cite{ST18} gives 
an example of a pair $h^1$ and $\widetilde{h}^2$ in the case when $G$ is the trivial group $\langle x_1 | x_1 \rangle $.
Note that, if we take the link $L$ complicated enough, the fundamental group of the cobordism $W(Y, H)$ may also be complicated.
As other examples of $h^1$ and $h^2$ in the case of $G\cong \langle x_1 | x_1 \rangle $, we can use the $1$- and $2$-handles which give the corks $W_n$, $W_{m,n}$ and $\overline{W}_n$ given in \cite{AY08}.

\subsubsection{Examples from $ 2$-knots}
Next, we consider homology $S^1\times S^3$'s obtained by surgeries along $2$-knots. 
Let $F$ be a 2-knot in $S^4$.
Let $\widetilde{F}: S^2 \times D^2 \to S^4$ be a tubular neighborhood of $F$.
Note that $\widetilde{F}$, which gives a framing of $F$, is unique up to isotopy.
We regard $\widetilde{F}$ as an attaching map of a $5$-dimensional $3$-handle $D^3 \times D^2$. Via the handle attaching along $\widetilde{F}$, we obtain a 5-dimensional compact cobordism 
\[
X(F):=   S^4 \times [0,1] \cup_{ \widetilde{F}} D^3 \times D^2.
\]
The boundary $\del X(F)$ consists of two connected components.
Let $V(F)$ denote the component which is not $S^{4} \times \{0\}$.
One can see that $V(F)$ is a homology $S^1\times S^3$. It is natural to ask when $V(F)$ admits a PSC metric. Our method gives a sufficient condition to obstruct PSC metric by using Seifert surfaces of $F$.
Now we recall the notion of Seifert surfaces of $F$. 
\begin{defi} 
We call a punctured orientable $3$-manifold $\Sigma'(F)$ a {\it Seifert surface} of $F$ if there exists an embedding $h: \Sigma'(F) \to S^4$ such that $h|_{\partial (\Sigma'(F))} = F$. 
\end{defi}
By capping the puncture of $\Sigma'(F)$, we obtain a closed 3-manifold $\Sigma(F)$.  Once we have a Seifert surface $\Sigma'(F)$ of $F$, we can construct an embedding $h': \Sigma(F) \to V(F)$ and choose a pair of orientations of the manifolds  $V(F)$ and $\Sigma(F)$ such that $\Sigma(F)$ is a cross-section of $V(F)$. We choose such orientations of  $V(F)$  and $\Sigma(F)$, and a spin structure $\mathfrak{s}(F)$ of $V(F)$.

\subsection{Examples of homology $3$-spheres with $\psi>0$}
To find explicit examples of $Y$ with $\psi(Y)>0$, which is a main assumption in \cref{cor psi examples}, we shall consider Brieskorn 3-manifolds. 
N.~Saveliev showed in \cite{Sa98} that $-\Sigma(p,q,pqm+1)$ for relatively prime numbers $p,q \geq 2$ and for odd $m$ bounds compact spin $4$-manifolds which violate the 10/8-inequalities.
This result gives us many examples of $Y$ with $\psi(Y)>0$.

\begin{ex}
\label{example}
Let $(p,q)$ be a pair of relatively prime numbers satisfying $A-B>1$ in Table~1 of \cite{Sa98}, where $A$ and $B$ denote some natural numbers defined in \cite{Sa98}.
Let $m$ be an odd positive integer, and $j$ be a positive integer.
In general, for an oriented manifold $N$, let $-N$ denote the same manifold with the reversed orientation.
In \cite{Sa98} Saveliev showed that $-\Sigma(p,q,pqm+1)$ bounds a compact simply connected smooth $4$-manifold whose intersection form is given by
\[
a (-E_8) \oplus b \left(\begin{matrix}0&1\\1&0\end{matrix}\right)
\]
for some $(a,b)$ satisfying $A\leq a$, $b \leq B$.
Besides this bounding, $\Sigma(p,q,pqm+1)$ bounds both negative and positive definite simply connected $4$-manifolds. (See \cite{Fr02}. Note that this holds also for even $m >0$.) 
It follows from this fact that the Fr\o yshov invariant satisfies
\begin{align}
\label{h is zero}
  h(-\Sigma(p,q,pqm+1))= 0
\end{align}
using the inequality appearing in Theorem~4 in \cite{Fr10}.
Thus we have
\[
\psi([\#_j (- \Sigma(p,q,pqm+1))])>0
\]
because of the equality \eqref{h is zero} and above Saveliev's bounding.

For example, for $j>0$ and odd $m$, let $Y$ be a $3$-manifold given as one of 
\begin{align}
\label{series of ex}
\begin{cases}
\#_j (- \Sigma(4,7,28m+1)),\\
\#_j (- \Sigma(4,15,60m+1)), \text{ and}\\
\#_j (- \Sigma(4,17,68m+1)).
\end{cases}
\end{align}
Then, we have $\psi([Y])>0$.
\end{ex} When $Y$ is an oriented integral homology $3$-sphere, $\epsilon(Y)$ is a homology cobordism invariant.
So we get a map
\begin{align}
\psi : \Theta^3 \rightarrow \Z,
\label{psi from theta}
\end{align}
where the group $\Theta^3$ is the homology cobordism group of oriented integer homology $3$-spheres.
Let us define the subsemigroup
\[
\Pi := \Set{[Y] \in \Theta^3 | \psi([Y])>0}
\]
of $\Theta^3$.
It is easy to show that, for $[Y] \in \Theta^3$ and $[Y']\in \Pi$, there exists sufficiently large $N>0$ such that $[Y] + N [Y'] \in \Pi$.
This property of $\psi$ gives the following example.

\begin{ex} \label{connected sum}
 Saveliev~\cite{Sa97} and Manolescu~\cite{Ma14} constructed the following spin boundings: 
\begin{itemize} 
\item $-\Sigma ( 2,q,2qk+1) = \partial ( -\left(\frac{q+1}{4}\right) E_8 \oplus  \left(\begin{matrix}0&1\\1&0\end{matrix}\right))$ and
\item $\Sigma (2,3,11) = \partial (-2E_8 \oplus 2  \left(\begin{matrix}0&1\\1&0\end{matrix}\right))$.
\end{itemize}

Based on these results, we can show 
\begin{align}\label{ineq2}
\psi (-\Sigma ( 2,q,2qk+1)) \geq \left(\frac{q+1}{4}\right)-2 \text{ and } \psi (\Sigma (2,3,11) ) \geq 0.
\end{align}
Moreover, for any homology $3$-sphere $Y$, there exists $N\gg 0$ such that 
\begin{itemize}
\item  $\psi (Y \# (-\Sigma ( 2,q,2q+1))) >0 $ and 
\item $\psi (Y \# ( \#_q (- \Sigma(4,7,29)))) >0 $
\end{itemize} 
for $q>N$. This is shown by using $\psi(-\Sigma(4,7,29))>0$ and \eqref{ineq2}.
\end{ex}

\subsection{Homology $S^{1} \times S^{3}$'s admitting no PSC metrics}
We give examples of homology $S^{1} \times S^{3}$'s for which we can show the non-existence of PSC metrics using our method, more precisely \cref{cor psi examples}.

\subsubsection{Results for $X(Y,H)$}
\begin{theo}\label{largest ex} Suppose $Y$ is an oriented homology $3$-sphere with $[Y] \in \prod$.
Then, for any algebraically canceling pair $H$ for $Y$, $X(Y, H)$ does not admit a PSC metric.
\end{theo}

\begin{proof}
This theorem follows from \cref{cor psi examples} and the fact that $X(Y,H)$ contains $Y$ as a cross-section.
\end{proof}

\begin{ex}
\label{ex: diagram ex} 
As a concrete example, we put $Y=-\Sigma(4,7,29)$, 
\[
G=\langle x_1, x_2, x_3 | x_1, x_2^5x_3^{-2}, (x_2x_3)^3x_3^{-2} \rangle,
\]
which is isomorphic to $A_5$, the alternating group of degree 5.
We take pairs of $1$- and $2$-handles $\{h^1_i\}_{1 \leq i \leq 3}$ and $\{\tilde{h}^2_i\}_{1 \leq i \leq 3}$ as follows. 
 First, we take a Seifert invariant of $Y$ as $b=0$, $b_1=-1$, $b_2=2$ and $b_3=-1$. (See, for example, Equation (1.5) on page 3~\cite{Sa02}.)
 Then $\Sigma(4,7, 29)$ is described as the surgery of the four component link which is obtained as the complement of the $1$-framed 2-handle and the 1-handle in \cref{diag1}, which is also given in  Fig.~1.1 on page 3~\cite{Sa02}.
 The pair $h^1_1$ and $\tilde{h}^2_1$ are given as the $1$-framed 2-handle and the 1-handle in \cref{diag1}.
\begin{figure}[htbp]
\begin{center}
\includegraphics[scale= 0.15]{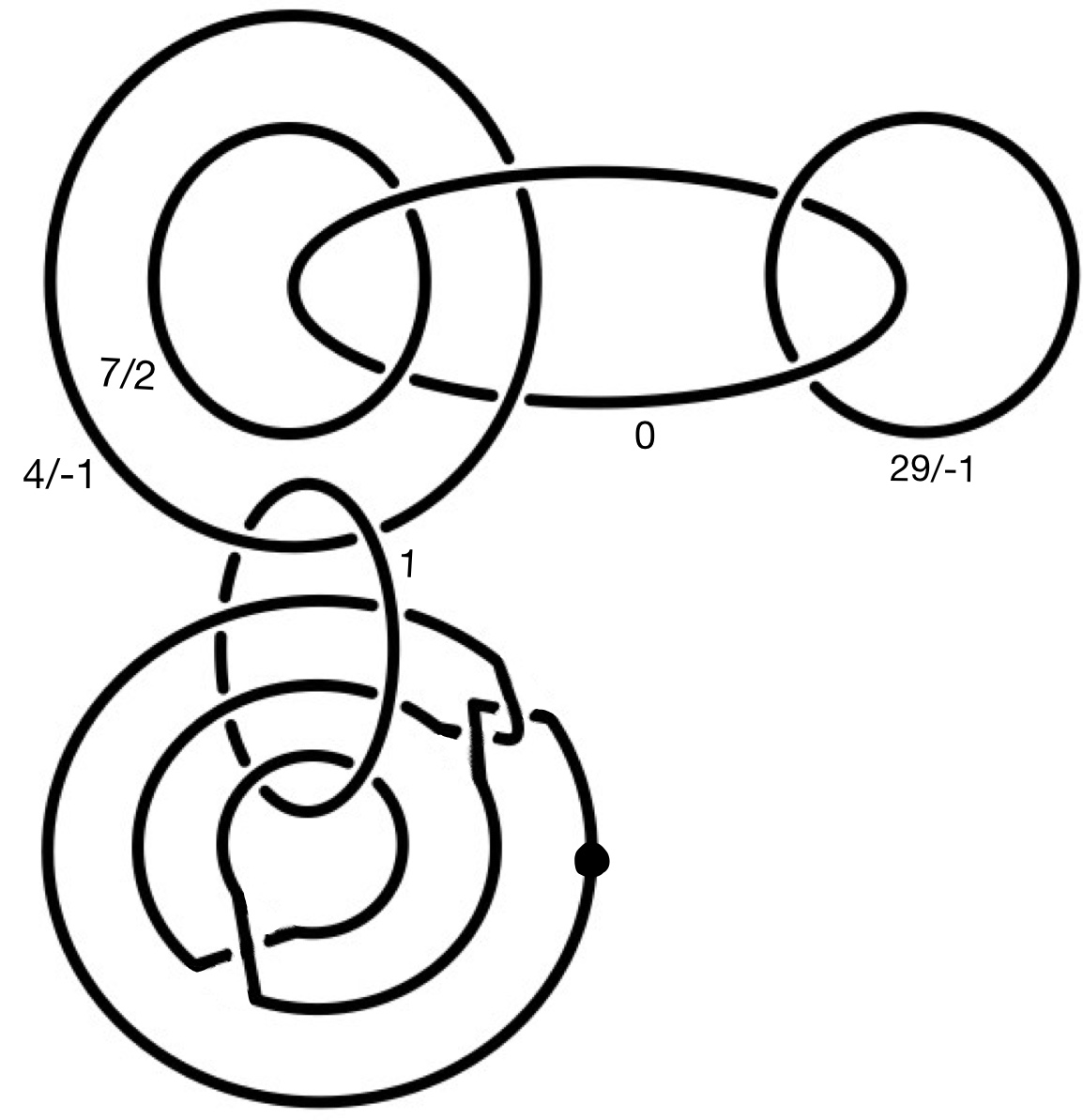}
\caption{A diagram of $\Sigma(4,7, 29)$ and $h^1_1$, $\tilde{h}^2_1$ }\label{diag1}
\end{center}
\end{figure}
We take other handles $h^1_2$, $h^1_3$, $h^2_2$ and $h^2_3$ representing $\langle x_2, x_3 | x_2^5x_3^{-2}, (x_2x_3)^3x_3^{-2}\rangle$ so that these are attached to an embedded disk $D^3 \subset Y$ which is disjoint from the diagram given in \cref{diag1}.
For $i=2,3$, let $\tilde{h}^{2}_{i}$ be the connected sum of $h^2_i$ with the trivial knot.
Then the pair $H:= (\{h^1_i\}_{1 \leq i \leq 3}, \{\tilde{h}^2_i\}_{1 \leq i \leq 3})$ is an algebraically canceling pair. 
It follows from \cref{connected sum} and
\cref{largest ex} that $X(Y, H)$ for these $Y$ and $H$ cannot admit a PSC metric. 

We note that our example $X(Y, H)$ is not diffeomorphic to the mapping torus of a closed $3$-manifold with respect to a self-diffeomorphism on it. 
It follows from the van Kampen theorem that the fundamental group of $W(Y, H)$ is isomorphic to the free product $\pi_1(Y) * G$. Put $W_0:= -W(Y, H)\cup_Y W(Y, H)$ and $W[0,n] := W_0 \cup_{Y'}W_1   \cup_{Y'} \cdots \cup_{Y'} W_n$, where $W_n$ is a copy of $W_0$ for each $n$. 
The van Kampen theorem implies
\begin{align}\label{infty}
\sup_{n\in \mathbb{N}} \rank (\pi_1 (W[0,n] ) )= \infty.
\end{align}
 Suppose $X(Y, H)$ is diffeomorphic to the mapping torus of a $3$-manifold $\widehat{Y}$ and a self-diffeomorphism $h : \widehat{Y} \to \widehat{Y}$. Then the $n$-fold cyclic covering space of  $X(Y, H)$ is given by the mapping torus $X_{h^n}(\widehat{Y})$ of $\widehat{Y}$ with respect to $h^n : \widehat{Y} \to \widehat{Y}$. Let $\tilde{Y} \subset X_{h^n}(\widehat{Y}) $ be a lift of $Y$ for the covering projection $X_{h^n}(\widehat{Y}) \to X(Y, H)$. The uniqueness of covering spaces implies that $\overline{X_{h^n}(\widehat{Y}) \setminus \tilde{Y}} \cong W[0,n-1]$.
However, one can see 
 \[
 \sup_{n\in \mathbb{N}} \rank (\pi_1 (\overline{X_{h^n}(\widehat{Y}) \setminus \tilde{Y}}  ) ) < \infty. 
 \]
 This contradicts to \eqref{infty}. 
\end{ex}

\subsubsection{Connected sum along $S^1$} 
Let $(X_1,\fraks_{1})$ and $(X_2,\fraks_{2})$ be spin rational homology $S^1\times S^3$'s.
We fix an embedding $f_i$ from $S^1$ to $X_i$ representing a fixed generator of $\Hom(\pi_{1}(X_{i}), \Z) \cong H^{1}(X_i;\mathbb{Z}) \cong \Z$ for $i=1,2$.
Assume that the pulled-back spin structures $f_{1}^{\ast}\fraks_{1}$ and $f_{2}^{\ast}\fraks_{2}$ on $S^{1}$ are isomorphic each other.
We also fix a tubular neighborhood of $f_i(S^1)$ and a trivialization $g_i: S^1 \times D^3 \to X_i$ of $f_i$ for each $i$.
If we choose an orientation reversing diffeomorphism $\xi $ from $S^1 \times S^2$ to itself, we obtain a diffeomorphism \[
\xi^*:= g_2|_{S^1\times \partial D^3} \circ \xi \circ (g_1|_{g_1(S^1\times \partial D^3)}) ^{-1} : g_1(S^1\times \partial D^3) \to g_2(S^1\times \partial D^3) .
\]
We define the {\it connected sum of $X_1$ and $X_2$ along $S^1$ via $\xi^*$} by 
\[
X_1 \#_{\xi^*} X_2 := (X_1 \setminus \text{int} (\im g_1)) \cup_{\xi^*} (X_2 \setminus \text{int}(\im g_2)) ,
\]
where $\text{int} (\im g_i)$ is the interior of $\im g_i$ in $X_i$ for $ i=1,2$.
We can show that $X_1 \#_{\xi^*} X_2$ is a rational homology $S^1 \times S^3$ and inherits a spin structure from $\fraks_{1}$ and $\fraks_{2}$.

 \cref{connected sum} implies the following fact.
 
\begin{cor}
\label{cor: connected sum along s1}
Let $(X,\fraks)$ be a spin rational homology $S^1 \times S^3$ which has some homology $3$-sphere as a cross-section.
Then there exists $N>0$ such that
\[
X \#_{\xi^*}  (S^1\times (-\Sigma ( 2,q,2q+1)))
\]
and
\[
X \#_{\xi^*}  (S^1\times (\#_q(- \Sigma(4,7,29))))
\]
do not admit PSC metrics for all $q>N$ and some $\xi^*$.
\end{cor}
\begin{proof} 
Let $Y$ be a homology $3$-sphere with $[Y]=1\in H_3(X;\Z)$.
Then \cref{connected sum} implies that there exists $N\gg0$ such that 
\begin{itemize}
\item  $\psi (Y \# (-\Sigma ( 2,q,2q+1))) >0 $ and 
\item $\psi (Y \# ( \#_q(- \Sigma(4,7,29)))) >0 $
\end{itemize} 
for $q>N$. Note that we can choose
\[
Y \# (-\Sigma ( 2,q,2q+1))\ \text{(resp.}\ Y \# ( \#_q(- \Sigma(4,7,29))))
\]
as a cross-section of
\[
X \#_{\xi^*}  (S^1\times (-\Sigma ( 2,q,2q+1)))\ \text{(resp.}\ X \#_{\xi^*}  (S^1\times (\#_q(- \Sigma(4,7,29))))
\]
for some $\xi^*$. Now we can use \cref{cor psi examples}, and this proves the \lcnamecref{cor: connected sum along s1}.
\end{proof}

\subsubsection{Results for surgeries of $2$-knots}
\begin{cor}
\label{Seifert}
If $\Sigma(F)$ is a rational homology $3$-sphere and $\psi ( \Sigma(F), \mathfrak{s}|_{\Sigma(F)} ) >0$, then $V(F)$ does not admit a PSC metric.
\end{cor}
\begin{proof}
This corollary follows from \cref{cor psi examples} and the fact that $V(F)$ contains $\Sigma(F)$ as a cross-section.
\end{proof}

For a given (1-)knot $K$ and an integer $k$, we have a $2$-knot $F_k (K)$ called a $k$-twisted spun knot of $K$. (See \cite{Z65}.) In \cite{Z65}, E.C. Zeeman showed the following fact. 
\begin{prop}[Zeeman~\cite{Z65}, THE MAIN THEOREM on page~486]\label{Zeeman}
For any knot $K$ and a non-zero integer $k$, we can choose the $k$-th branched covering space $\Sigma_k(K)$ of $K$ as a Seifert surface of $F_k (K)$.
\end{prop}
By combining \cref{Seifert} and \cref{Zeeman}, we obtain the following corollary. 
\begin{cor} Let $K$ be a knot and $k$ be a non-zero positive integer.
If we have $\psi (\Sigma_k(K), \mathfrak{s})>0$ for any spin structure $\mathfrak{s}$ on $\Sigma_k(K)$, then $V(F_k(K))$ does not admit a PSC metric.
\end{cor}
For a pairwise relatively prime triple $(p,q,r)$, it is shown that the Brieskorn 3-manifold $\Sigma(p,q,r)$ is the cyclic branched $r$-hold covering space of the $(p,q)$-torus knot $T(p,q)$. Thus we obtain the following examples. 
\begin{ex} 
Let $T(4,7)$ be the $(4,7)$-torus knot.
For an odd positive integer $m$, we consider the $(28m+1)$-fold twisted spun knot $ F_{28m+1} (T(4,7))$. 
By \cref{example}, one can show that $V(F_{28m+1} (T(4,7)))$ does not admit a PSC metric. 
\end{ex}

\subsection{Comparison with other methods} 
We compare our method, particularly \cref{cor psi examples} and \cref{largest ex}, with the Dirac obstruction, enlargeability method, Schoen--Yau's method, and Lin's formula. 
\subsubsection{Dirac obstruction} 
For a given closed spin $n$-manifold $X$, Rosenberg~\cite{R86} defined an element $A(X) \in KO_n ( C^*_\R ( \pi_1 (X) ))$ such that if $X$ admits a PSC metric, then $A(X)=0$. Here, $C^*_\R ( \pi_1 (X) )$ is the real group $C^*$-algebra of $\pi_1 (X)$ and $KO_n ( -)$ is the real $K$-homology of a $C^*$-algebra. 
If one tries to prove \cref{largest ex} using this method, one needs to calculate the class
\[
A(X(Y, H) ) \in KO_4 ( C^*_\R ( \pi_1 (X(Y, H)) )).
\]
However, since $\pi_{1}(X(Y, H))$ may be complicated in general,
 it seems difficult to calculate $A(X(Y, H)) \in KO_4 ( C^*_\R ( \pi_1 (X(Y, H)) ))$ and deduce \cref{largest ex}. 
\subsubsection{Enlargeability}Gromov and Lawson \cite{GL80} introduced the notion of enlargeability which obstructs PSC metrics. 
One can give a large class of homology $S^{1} \times S^{3}$'s for which PSC metrics are obstructed by enlargeability as follows.
Suppose that a $3$-manifold $Y$ is enlargeable. Then, the mapping torus $X_h(Y)$ of $Y$ with respect to a self-diffeomorphism $h$ on $Y$ does not admit a PSC metric since the $\Z$-covering of $X_h(Y)$ is enlargeable.
However, our obstruction can be valid also for a homology $S^{1} \times S^{3}$ which cannot be the mapping torus of a $3$-manifold and a self-diffeomorphism as we have seen in \cref{ex: diagram ex}.

We note that Hanke and Schick (\cite{HS06}) gives a relation between the Dirac obstruction and enlargeablity as follows: if a closed spin $n$-manifold $X$ is enlargeable, then $0\neq A(X) \in KO_n ( C^*_\R ( \pi_1 (X) ))$.

\subsubsection{Schoen--Yau's method}
Schoen--Yau's result~\cite{SY79} implies the following fact.
Let $X$ be a spin rational homology $S^{1}\times S^{3}$ and assume that $X$ admits a PSC metric.
Then there exists a cross-section $Y'$ of $X$ such that $Y'$ also admits a PSC metric.
As a consequence of G.~Perelman's theory (for example, see Theorem 1.29 in \cite{Lee19}), one can show that $Y'$ is diffeomorphic to a $3$-manifold of the form
\begin{align}\label{SY1}
(\#_{i=1}^{m} S^{3}/\Gamma_{i}) \#  (\#_{j=1}^{n} (S^{1} \times S^{2})_j),
\end{align}
where $m,n \geq 0$, each $\Gamma_i$ is a finite subgroup of $SO(4)$ which acts on $S^3$ freely, and $\#_{j=1}^{n} (S^{1} \times S^{2})_j$ is the connected sum of $n$ copies of $S^1 \times S^2$.
Here, if $m=0$, let us replace $(\#_{i=1}^{m} S^{3}/\Gamma_{i})$ with $S^{3}$, and similarly $(\#_{j=1}^{n} (S^{1} \times S^{2})_j)$ with $S^{3}$ if $n=0$.
 For above $X$, we define
 \[
 n(X):= \min  \Set{ b_1 (Y' )  | 
\begin{matrix}
\text{$Y'$ is an orientable cross-section of $X$,} \\
\text{and $Y'$ admits a PSC metric.}
\end{matrix}
}.
 \]
 If $n(X)=0$, then one can obtain a $10/8$-type inequality 
\[
b^+(M) \geq -\frac{\sigma (M) }{8} + h(Y, \mathfrak{t} ) -1
\]
using \cite{Ma14} and the following facts:
 \begin{itemize}
  \item By cutting the total space of a $\Z$-covering space of $X$ along the cross-sections, we can show that there is a rational spin homology cobordism between $Y$ and $\#_{i=1}^{m}S^{3}/\Gamma_{i}$. 
  \item For a spin spherical $3$-manifold $(Y', \mathfrak{t} )$, 
\[
\kappa (Y', \mathfrak{t} ) = \lambda_{SW} ( Y' \times S^1 ,\mathfrak{t} \times ({\rm trivial}) )   = -  h(Y',  \mathfrak{t})
\]
holds. (See Theorem A in \cite{LRS17}.)
  \end{itemize}

 As in our situation, we suppose that  $X$ has a rational homology $3$-sphere $Y$ as a cross-section.
It seems difficult to obstruct the existence of PSC metric for $(X,Y)$ satisfying the following conditions.
Suppose that we have an element $[Y]$ in the rational homology cobordism group $\Theta^3_\Q$ such that $[Y]$ does not belong to the subgroup generated by all spherical $3$-manifolds.
The authors do not know whether such $Y$ exists.
Let us summarize this as a problem:
\begin{problem}Let $H_{\text{sp}}$ be the subgroup in $\Theta^3_\Q$ generated by spherical $3$-manifolds.
 Is the group $\Theta^3_\Q/ H_{\text{sp}}$ trivial? 
\end{problem}
If we have such an example of $Y$, one can show that for any $X$ which is a rational homology $S^1 \times S^3$ containing $Y$ as a cross-section and which has a PSC metric, $n(X)>0$ holds. Therefore we cannot apply the result of \cite{Ma14} immediately.

We also note that Schoen--Yau (Theorem~6 in \cite{SY86}) stated that any closed aspherical $4$-manifold cannot have a PSC metric.
However, for a given $4$-manifold, it is not obvious to see if the $4$-manifold is aspherical, and there seems to be no reason to expect that all $X(Y, H)$ are aspherical.
\subsubsection{Lin's formula}
Lin~\cite{L16} showed the equality \eqref{obst} under the assumption that $X$ admits a PSC metric and $Y$ is a cross-section of $X$.
On the other hand, the mod $2$ reduction of $\lambda_{SW}(X,\fraks)$ coincides with the Rochlin invariant $\mu(Y,\frakt)$ of $Y$~\cite{MRS11}.
Therefore the equality
\begin{align}\label{JLeq1}
h(Y) \equiv \mu (Y) \mod 2
\end{align}
holds if  $X$ admits a PSC metric and $Y$ is a cross-section of $X$, hence this equality \eqref{JLeq1} also gives an obstruction to PSC metric on $X$.
For example, let $X$ be a spin rational homology $S^1\times S^3$ which has $\#_j (\pm \Sigma(p,q,pqm+1))$ as a cross-section for $p,q$ and $m$ satisfying 
 \begin{align}\label{Jlinex}
 -\frac{1}{24}m(p^2-1) (q^2-1) \equiv 1 \mod 2 \text{ and } j \equiv 1 \mod 2.
 \end{align}
Then one can deduce that $X$ does not admit a PSC metric.
 The number $-\frac{1}{24}m(p^2-1) (q^2-1)$ is equal to the Casson invariant $\lambda (\Sigma(p,q,pqm+1))$ (see Example 3.30 in \cite{Sa02}).
Since the Fr\o yshov invariant of $\Sigma(p,q,pqm+1)$ is equal to 0 for all $p$, $q$ and $m$ as we have seen in \eqref{h is zero}, $\#_j (\pm \Sigma(p,q,pqm+1))$ does not satisfy \eqref{JLeq1} under the condition \eqref{Jlinex}. This calculation shows that Lin's method can produce many $4$-manifolds which do not admit PSC metrics.
On the other hand, note that one cannot show the non-existence of a PSC metric on any rational homology $S^1\times S^3$ which has $Y$ given in \eqref{series of ex} as a cross-section using the obstruction obtained from the equality \eqref{JLeq1} since $h(Y)\equiv \mu(Y)\equiv 0$ ${\rm mod}\ 2$ holds.
(Note that our method can be used even in this case, as in \cref{example,ex: diagram ex}.)

Next, we compare our method with Lin's formula for mapping tori of homology $3$-spheres.  Let $(q,r)$ be a pair of relatively prime odd numbers. Let $T(q,r)$ denote the $(q,r)$-torus knot. Note that the double branched cover of $T(q,r)$ is the Seifert manifold $\Sigma (2,q,r)$.
It is known that the signature of $T(q,r)$ is equal to $8\lambda (\Sigma (2,q,r))$ (see Example 5.10 of \cite{Sa02}). 
Let $K$ be $(-T(q,2qk+1)) \# l T(3,11)$ where $q, k$ and $l$ are positive integers satisfying  
\[
q \equiv 3 \mod 4,  \ k \equiv 1 \mod 2,\ l>0 \text{ and } kq^2 = 16l+k.
\]
The double branched cover $\Sigma(K)$ of $K$ is $-\Sigma ( 2,q,2qk+1)\#  l \Sigma (2,3,11)$. (It is known that $\Sigma(K^*)=-\Sigma(K)$ and $\Sigma(K\# J)= \Sigma(K) \# \Sigma(J)$, where $K^*$ is the reflection of $K$.) Let $\tau$ be the involution of the branched cover. We set $X(K)$ as the mapping torus of $\tau$.
In Theorem~C of \cite{LRS18}, Lin--Ruberman--Saveliev showed 
\[
-\lambda_{SW} (X(K) )= \frac{\text{sign} (K)}{8},
\]
where $\text{sign} (K)$ is the signature of $K$.
Therefore we have 
\begin{align*}
\lambda_{SW} (X(K) )& = -  \frac{\text{sign} (K)}{8} \\
& =  \frac{\text{sign} (T(q,2qk+1) )}{8} -   \frac{l\text{sign} (T(3,11))}{8} \\
& =  \lambda (\Sigma ( 2,q,2qk+1)) - l \lambda (\Sigma (2,3,11 )) \\ 
& =  -\frac{1}{8}k (q^2-1)+ 2l = 0 
\end{align*}
by the choice of $k$, $l$ and $q$.
We take $Y(K)=-\Sigma ( 2,q,2qk+1)\#  l \Sigma (2,3,11)$ as a cross-section of $X(K)$.
Since $h(\Sigma(p,q,pqk+1))=0$ for a pair of relatively prime numbers $(p,q)$ and a positive integer $k$, we get $h(Y(K))=0$. Therefore the pair $(X(K), Y(K))$ satisfies \eqref{obst}. Thus, we cannot obstruct PSC metrics of $X(K)$ using \eqref{obst}.
However, we have $\psi (Y(K))$ is positive by \eqref{ineq2} if $q \geq 11$. Thus, we can see that $X(K)$ does not admit a PSC metric by using \cref{cor psi examples}.

\begin{rem}
Lin's method~\cite{L16} and ours obstruct PSC metric on homology $S^{1}\times S^{3}$'s only in terms of topological properties of cross-sections of them.
In our case, the obstruction is dominated by the subsemigroup
$\Pi$ of $\Theta^3$, so this subsemigroup might be an interesting object to study.
In this \lcnamecref{examples}, we gave many examples of elements of $\Pi$.
Moreover, as we explained, for any element $[Y] \in \Theta^3$ and $[Y']\in \Pi$, there exists a natural number $N \gg 0$ such that $[Y] \# N [Y'] \in \Pi$.
This may suggest that $\Pi$ is a large subsemigroup of $\Theta$, and therefore it is natural to ask the following question:

\begin{problem}
Study the subsemigroup $\Pi$.
For example, how large is $\Pi$ in $\Theta^3$?
More precisely, is there a sequence of elements of $\Pi$ which generates $\Z^{\infty}$ in $\Theta^3$?
\end{problem}
\end{rem}
  
\bibliographystyle{plain}
\bibliography{tex}

\begin{bibdiv}
\begin{biblist}

\bib{AY08}{article}{
      author={Akbulut, Selman},
      author={Yasui, Kouichi},
       title={Corks, plugs and exotic structures},
        date={2008},
     journal={J. G\"{o}kova Geom. Topol. GGT},
      volume={2},
       pages={40\ndash 82},
      review={\MR{2466001}},
}

\bib{ABS64}{article}{
      author={Atiyah, M.~F.},
      author={Bott, R.},
      author={Shapiro, A.},
       title={Clifford modules},
        date={1964},
        ISSN={0040-9383},
     journal={Topology},
      volume={3},
      number={suppl. 1},
       pages={3\ndash 38},
         url={https://doi.org/10.1016/0040-9383(64)90003-5},
      review={\MR{0167985}},
}

\bib{BF04}{article}{
      author={Bauer, Stefan},
      author={Furuta, Mikio},
       title={A stable cohomotopy refinement of {S}eiberg-{W}itten invariants.
  {I}},
        date={2004},
        ISSN={0020-9910},
     journal={Invent. Math.},
      volume={155},
      number={1},
       pages={1\ndash 19},
         url={https://doi.org/10.1007/s00222-003-0288-5},
      review={\MR{2025298}},
}

\bib{Do83}{article}{
      author={Donaldson, S.~K.},
       title={An application of gauge theory to four-dimensional topology},
        date={1983},
        ISSN={0022-040X},
     journal={J. Differential Geom.},
      volume={18},
      number={2},
       pages={279\ndash 315},
         url={http://projecteuclid.org/euclid.jdg/1214437665},
      review={\MR{710056}},
}

\bib{DK90}{book}{
      author={Donaldson, S.~K.},
      author={Kronheimer, P.~B.},
       title={The geometry of four-manifolds},
      series={Oxford Mathematical Monographs},
   publisher={The Clarendon Press, Oxford University Press, New York},
        date={1990},
        ISBN={0-19-853553-8},
        note={Oxford Science Publications},
      review={\MR{1079726}},
}

\bib{Fr02}{article}{
      author={Fr{\o}yshov, Kim~A.},
       title={Equivariant aspects of {Y}ang-{M}ills {F}loer theory},
        date={2002},
        ISSN={0040-9383},
     journal={Topology},
      volume={41},
      number={3},
       pages={525\ndash 552},
         url={http://dx.doi.org/10.1016/S0040-9383(01)00018-0},
      review={\MR{1910040}},
}

\bib{Fr10}{article}{
      author={Fr{\o}yshov, Kim~A.},
       title={Monopole {F}loer homology for rational homology 3-spheres},
        date={2010},
        ISSN={0012-7094},
     journal={Duke Math. J.},
      volume={155},
      number={3},
       pages={519\ndash 576},
         url={https://doi.org/10.1215/00127094-2010-060},
      review={\MR{2738582}},
}

\bib{Fu01}{article}{
      author={Furuta, M.},
       title={Monopole equation and the {$\frac{11}{8}$}-conjecture},
        date={2001},
        ISSN={1073-2780},
     journal={Math. Res. Lett.},
      volume={8},
      number={3},
       pages={279\ndash 291},
         url={https://doi.org/10.4310/MRL.2001.v8.n3.a5},
      review={\MR{1839478}},
}

\bib{FK05}{article}{
      author={Furuta, M.},
      author={Kametani, Y.},
       title={Equivariant maps between sphere bundles over tori and
  {KO}-degree},
        date={2005},
      eprint={arXiv:math/0502511},
}

\bib{GL80}{article}{
      author={Gromov, Mikhael},
      author={Lawson, H.~Blaine, Jr.},
       title={The classification of simply connected manifolds of positive
  scalar curvature},
        date={1980},
        ISSN={0003-486X},
     journal={Ann. of Math. (2)},
      volume={111},
      number={3},
       pages={423\ndash 434},
         url={https://doi.org/10.2307/1971103},
      review={\MR{577131}},
}

\bib{HS06}{article}{
      author={Hanke, B.},
      author={Schick, T.},
       title={Enlargeability and index theory},
        date={2006},
        ISSN={0022-040X},
     journal={J. Differential Geom.},
      volume={74},
      number={2},
       pages={293\ndash 320},
         url={http://projecteuclid.org/euclid.jdg/1175266206},
      review={\MR{2259056}},
}

\bib{Ka18}{article}{
      author={Kametani, Yukio},
       title={Spin structures and the divisibility of {E}uler classes},
        date={2018},
      eprint={arXiv:1809.04045},
}

\bib{KM07}{book}{
      author={Kronheimer, Peter},
      author={Mrowka, Tomasz},
       title={Monopoles and three-manifolds},
      series={New Mathematical Monographs},
   publisher={Cambridge University Press, Cambridge},
        date={2007},
      volume={10},
        ISBN={978-0-521-88022-0},
         url={https://doi.org/10.1017/CBO9780511543111},
      review={\MR{2388043}},
}

\bib{L95}{book}{
      author={Lang, Serge},
       title={Differential and {R}iemannian manifolds},
     edition={Third},
      series={Graduate Texts in Mathematics},
   publisher={Springer-Verlag, New York},
        date={1995},
      volume={160},
        ISBN={0-387-94338-2},
         url={https://doi.org/10.1007/978-1-4612-4182-9},
      review={\MR{1335233}},
}

\bib{Lee19}{book}{
      author={Lee, Dan~A.},
       title={Geometric relativity},
      series={Graduate Studies in Mathematics},
   publisher={American Mathematical Society, Providence, RI},
        date={2019},
      volume={201},
        ISBN={978-1-4704-5081-6},
      review={\MR{3970261}},
}

\bib{Lin15}{article}{
      author={Lin, Jianfeng},
       title={Pin(2)-equivariant {KO}-theory and intersection forms of spin
  4-manifolds},
        date={2015},
        ISSN={1472-2747},
     journal={Algebr. Geom. Topol.},
      volume={15},
      number={2},
       pages={863\ndash 902},
         url={https://doi.org/10.2140/agt.2015.15.863},
      review={\MR{3342679}},
}

\bib{L16}{article}{
      author={Lin, Jianfeng},
       title={The {S}eiberg-{W}itten equations on end-periodic manifolds and an
  obstruction to positive scalar curvature metrics},
        date={2019},
        ISSN={1753-8416},
     journal={J. Topol.},
      volume={12},
      number={2},
       pages={328\ndash 371},
         url={https://doi.org/10.1112/topo.12090},
      review={\MR{3911569}},
}

\bib{LRS18}{article}{
      author={Lin, Jianfeng},
      author={Ruberman, Daniel},
      author={Saveliev, Nikolai},
       title={On the {F}r\o yshov invariant and monopole {L}efschetz number},
        date={2018},
      eprint={arXiv:1802.07704},
}

\bib{LRS17}{article}{
      author={Lin, Jianfeng},
      author={Ruberman, Daniel},
      author={Saveliev, Nikolai},
       title={A splitting theorem for the {S}eiberg-{W}itten invariant of a
  homology {$S^1\times S^3$}},
        date={2018},
        ISSN={1465-3060},
     journal={Geom. Topol.},
      volume={22},
      number={5},
       pages={2865\ndash 2942},
         url={https://doi.org/10.2140/gt.2018.22.2865},
      review={\MR{3811774}},
}

\bib{Ma14}{article}{
      author={Manolescu, Ciprian},
       title={On the intersection forms of spin four-manifolds with boundary},
        date={2014},
        ISSN={0025-5831},
     journal={Math. Ann.},
      volume={359},
      number={3-4},
       pages={695\ndash 728},
         url={https://doi.org/10.1007/s00208-014-1010-1},
      review={\MR{3231012}},
}

\bib{Ma82}{article}{
      author={Matsumoto, Yukio},
       title={On the bounding genus of homology {$3$}-spheres},
        date={1982},
        ISSN={0040-8980},
     journal={J. Fac. Sci. Univ. Tokyo Sect. IA Math.},
      volume={29},
      number={2},
       pages={287\ndash 318},
      review={\MR{672065}},
}

\bib{MS12}{book}{
      author={McDuff, Dusa},
      author={Salamon, Dietmar},
       title={{$J$}-holomorphic curves and symplectic topology},
     edition={Second},
      series={American Mathematical Society Colloquium Publications},
   publisher={American Mathematical Society, Providence, RI},
        date={2012},
      volume={52},
        ISBN={978-0-8218-8746-2},
      review={\MR{2954391}},
}

\bib{M96}{book}{
      author={Morgan, John~W.},
       title={The {S}eiberg-{W}itten equations and applications to the topology
  of smooth four-manifolds},
      series={Mathematical Notes},
   publisher={Princeton University Press, Princeton, NJ},
        date={1996},
      volume={44},
        ISBN={0-691-02597-5},
      review={\MR{1367507}},
}

\bib{Mo03}{incollection}{
      author={Morgan, John~W.},
       title={Definition of the {S}eiberg-{W}itten ({SW}) invariants of
  4-manifolds},
        date={2003},
   booktitle={Low dimensional topology},
      series={New Stud. Adv. Math.},
      volume={3},
   publisher={Int. Press, Somerville, MA},
       pages={1\ndash 11},
      review={\MR{2052242}},
}

\bib{MRS11}{article}{
      author={Mrowka, Tomasz},
      author={Ruberman, Daniel},
      author={Saveliev, Nikolai},
       title={Seiberg-{W}itten equations, end-periodic {D}irac operators, and a
  lift of {R}ohlin's invariant},
        date={2011},
        ISSN={0022-040X},
     journal={J. Differential Geom.},
      volume={88},
      number={2},
       pages={333\ndash 377},
         url={http://projecteuclid.org/euclid.jdg/1320067650},
      review={\MR{2838269}},
}

\bib{R86}{article}{
      author={Rosenberg, Jonathan},
       title={{$C^\ast$}-algebras, positive scalar curvature, and the {N}ovikov
  conjecture. {III}},
        date={1986},
        ISSN={0040-9383},
     journal={Topology},
      volume={25},
      number={3},
       pages={319\ndash 336},
         url={https://doi.org/10.1016/0040-9383(86)90047-9},
      review={\MR{842428}},
}

\bib{Ru98}{book}{
      author={Ruan, Yongbin},
       title={Virtual neighborhoods and the monopole equations},
      series={First Int. Press Lect. Ser., I},
   publisher={Int. Press, Cambridge, MA},
        date={1998},
      review={\MR{1635698}},
}

\bib{RS07}{article}{
      author={Ruberman, Daniel},
      author={Saveliev, Nikolai},
       title={Dirac operators on manifolds with periodic ends},
        date={2007},
        ISSN={1935-2565},
     journal={J. G\"okova Geom. Topol. GGT},
      volume={1},
       pages={33\ndash 50},
      review={\MR{2386534}},
}

\bib{MR2284407}{article}{
      author={Sasahira, Hirofumi},
       title={Spin structures on {S}eiberg-{W}itten moduli spaces},
        date={2006},
        ISSN={1340-5705},
     journal={J. Math. Sci. Univ. Tokyo},
      volume={13},
      number={3},
       pages={347\ndash 363},
      review={\MR{2284407}},
}

\bib{ST18}{article}{
      author={Sato, Kouki},
      author={Taniguchi, Masaki},
       title={Rational homology 3--spheres and simply connected definite
  bounding},
        date={2020},
        ISSN={1472-2747},
     journal={Algebr. Geom. Topol.},
      volume={20},
      number={2},
       pages={865\ndash 882},
         url={https://doi.org/10.2140/agt.2020.20.865},
      review={\MR{4092313}},
}

\bib{Sa97}{incollection}{
      author={Saveliev, Nikolai},
       title={On the homology cobordism group of homology {$3$}-spheres},
        date={1997},
   booktitle={Geometry, topology and physics ({C}ampinas, 1996)},
   publisher={de Gruyter, Berlin},
       pages={245\ndash 257},
      review={\MR{1605244}},
}

\bib{Sa98}{article}{
      author={Saveliev, Nikolai},
       title={Dehn surgery along torus knots},
        date={1998},
        ISSN={0166-8641},
     journal={Topology Appl.},
      volume={83},
      number={3},
       pages={193\ndash 202},
         url={https://doi.org/10.1016/S0166-8641(97)00109-0},
      review={\MR{1606386}},
}

\bib{Sa02}{book}{
      author={Saveliev, Nikolai},
       title={Invariants of homology 3-spheres},
   publisher={Encyclopaedia of Mathematical Sciences, vol. 140 (Springer,
  Berlin).},
        date={2002},
}

\bib{SY79}{article}{
      author={Schoen, R.},
      author={Yau, S.~T.},
       title={On the structure of manifolds with positive scalar curvature},
        date={1979},
        ISSN={0025-2611},
     journal={Manuscripta Math.},
      volume={28},
      number={1-3},
       pages={159\ndash 183},
         url={https://doi.org/10.1007/BF01647970},
      review={\MR{535700}},
}

\bib{SY86}{incollection}{
      author={Schoen, Richard},
      author={Yau, Shing-Tung},
       title={The structure of manifolds with positive scalar curvature},
        date={1987},
   booktitle={Directions in partial differential equations ({M}adison, {WI},
  1985)},
      series={Publ. Math. Res. Center Univ. Wisconsin},
      volume={54},
   publisher={Academic Press, Boston, MA},
       pages={235\ndash 242},
      review={\MR{1013841}},
}

\bib{Sto92}{article}{
      author={Stolz, Stephan},
       title={Simply connected manifolds of positive scalar curvature},
        date={1992},
        ISSN={0003-486X},
     journal={Ann. of Math. (2)},
      volume={136},
      number={3},
       pages={511\ndash 540},
         url={https://doi.org/10.2307/2946598},
      review={\MR{1189863}},
}

\bib{T87}{article}{
      author={Taubes, Clifford~Henry},
       title={Gauge theory on asymptotically periodic {$4$}-manifolds},
        date={1987},
        ISSN={0022-040X},
     journal={J. Differential Geom.},
      volume={25},
      number={3},
       pages={363\ndash 430},
         url={http://projecteuclid.org/euclid.jdg/1214440981},
      review={\MR{882829}},
}

\bib{Ve14}{article}{
      author={Veloso, Diogo},
       title={Seiberg--{W}itten theory on 4-manifolds with periodic ends},
        date={2018},
      eprint={arXiv:1807.11930},
}

\bib{Z65}{article}{
      author={Zeeman, E.~C.},
       title={Twisting spun knots},
        date={1965},
        ISSN={0002-9947},
     journal={Trans. Amer. Math. Soc.},
      volume={115},
       pages={471\ndash 495},
         url={https://doi.org/10.2307/1994281},
      review={\MR{0195085}},
}

\end{biblist}
\end{bibdiv}

\end{document}